\documentclass[a4paper]{article}

\usepackage[utf8]{inputenc}
\usepackage[T1]{fontenc}
\usepackage{amssymb,amsmath,amsthm,amscd,mathrsfs}
\usepackage[all]{xy}
\usepackage{color}

\usepackage{microtype}
\usepackage[margin=3cm]{geometry}
\usepackage{times}
\usepackage[colorlinks,pagebackref,hypertexnames=false,pdftitle={Construction of G2-instantons via twisted connected sums}]{hyperref}
\usepackage{amsrefs}

\usepackage{enumerate}
\usepackage{multicol}
\RequirePackage{xspace}
\usepackage{enumitem}
\usepackage[normalem]{ulem}

\newtheorem{thm}{Theorem}[section]
\newtheorem{prop}[thm]{Proposition}
\newtheorem{lemme}[thm]{Lemma}
\newtheorem{cor}[thm]{Corollary}

\theoremstyle{definition}
\newtheorem{defi}[thm]{Definition}
\newtheorem{ex}[thm]{Example}
\newtheorem{rmk}[thm]{Remark}
\newtheorem{summ}[thm]{Summary}

\DeclareMathOperator{\Z}{\mathbb{Z}}
\DeclareMathOperator{\R}{\mathbb{R}}

\DeclareMathOperator{\Pic}{Pic}
\DeclareMathOperator{\rk}{rk}

\DeclareMathOperator{\Ext}{Ext}

\DeclareMathOperator{\Extr}{\mathscr{E}\textit{xt}}
\DeclareMathOperator{\image}{im}

\DeclareMathOperator{\End}{End}

\DeclareMathOperator{\Endr}{\mathscr{E}\textit{nd}}
\DeclareMathOperator{\tr}{tr}
\DeclareMathOperator{\C}{\mathbb{C}}

\DeclareMathOperator{\dif}{d}
\DeclareMathOperator{\md}{mod}

\DeclareMathOperator{\res}{res}
\DeclareMathOperator{\Hol}{Hol}
\DeclareMathOperator{\Bl}{Bl}
\DeclareMathOperator{\Div}{Div}

\newcommand{\ignore}[1]{}
\newcommand{\gtmfd}{$G_{2}$\nobreakdash-\hspace{0pt}manifold}
\newcommand{\CP}{\mathbb{P}}
\newcommand{\PP}{\mathbb{P}}
\newcommand{\mcal}{\mathcal{M}}
\newcommand{\mscr}{\mathscr{M}}
\newcommand{\ccal}{\mathscr{C}}

\newcommand{\ie}{\emph{i.e.} }
\newcommand{\eg}{\emph{e.g.} }
\newcommand{\cf}{\emph{cf.} }

\newcommand{\wt}{\widetilde}
\newcommand{\kd}{S}
\newcommand{\nres}{R}
\newcommand{\sff}{\mathcal{Y}}
\newcommand{\into}{\hookrightarrow}
\DeclareMathOperator{\amp}{Amp}
\newcommand{\hk}{hyper-K\"ahler\xspace}
\DeclareMathAlphabet{\df}{U}{eus}{m}{n}
\DeclareMathAlphabet{\matheur}{U}{eur}{m}{n}
\newcommand{\hkr}{\matheur{r}}
\newcommand{\kclass}{\matheur{k}}
\newcommand{\gtstr}{$G_{2}$\nobreakdash-\hspace{0pt}structure}
\newcommand{\degl}{\deg_L}

\newcommand{\qand}{\quad\text{and}\quad}

\newcommand{\eq}[1][r]
{\ar@<-3pt>@{-}[#1]
\ar@<-1pt>@{}[#1]|<{}="gauche"
\ar@<+0pt>@{}[#1]|-{}="milieu"
\ar@<+1pt>@{}[#1]|>{}="droite"
\ar@/^2pt/@{-}"gauche";"milieu"
\ar@/_2pt/@{-}"milieu";"droite"}

\newcommand{\incl}[1][r]
  {\ar@<-0.2pc>@{^(-}[#1] \ar@<+0.2pc>@{-}[#1]} 

\ignore{)} 

\author{Gr\'egoire \textsc{Menet}, Johannes \textsc{Nordstr\"om}, Henrique N. \textsc{S\'a~Earp}}

\begin{document}

\date{\today}
\title{\bf Construction of $\rm G_2$-instantons via twisted connected sums}

\maketitle

\begin{abstract}
We propose a method to construct  $\rm G_2$--instantons over a compact twisted connected sum
$\rm G_2$--manifold, applying a gluing result of S\'a Earp and Walpuski
to instantons  over a pair of $7$--manifolds
with a tubular end.
In our example, the moduli spaces of the ingredient instantons are
non-trivial, and their images in the moduli space over the asymptotic cross-section
 K3 surface
intersect transversely.
Such a pair of asymptotically stable holomorphic bundles is obtained using a twisted version of the
Hartshorne-Serre construction, which can be adapted to produce other examples.
Moreover, their deformation theory and asymptotic behaviour are explicitly understood, results which
may be of independent interest.
\end{abstract}

\tableofcontents

\newpage
\section{Introduction}
We address the existence problem of $\rm G_2$--instantons over twisted connected
sums as formulated by the third author and Walpuski in  \cite{Henrique0}, and we produce
the first examples to date of solutions obtained by a nontrivially \emph{transversal}
gluing process.

Recall that a \emph{$\rm G_2$--manifold} $(X,g_{\phi})$ is a Riemannian $7$--manifold together with a torsion-free \emph{$\rm G_2$--structure}, that is, a non-degenerate closed $3$--form $\phi$ satisfying a certain non-linear partial differential equation;
in particular, $\phi$ induces a Riemannian metric $g_\phi$ with
$\Hol(g_{\phi})\subset \rm G_2$.
A \emph{$\rm G_2$--instanton} is a connection $A$ on some $G$--bundle $E\to X$ such that $F_A\wedge*\phi=0$.
Such solutions have a well-understood elliptic deformation theory of index
$0$ \cite{Henrique1}, and some form of `instanton count' of their
moduli space is expected to yield new invariants of $7$--manifolds, much
in the same vein as the Casson invariant and instanton Floer homology from flat connections on $3$--manifolds \cite{Donaldson1, DT}. While some important analytical groundwork has been established towards that goal \cite{Tian},
major compactification issues remain and this suggests that a thorough understanding of
the general theory might currently have to be postponed in favour of  exploring a good number of functioning examples. The present paper proposes a method to construct a potentially large number of such instances.

Readers interested in a more detailed account of instanton theory on $\rm G_2$--manifolds are kindly
 referred to the introductory sections of \cite{Henrique2,Henrique0} and works cited therein.

\subsection{\texorpdfstring{$\rm G_{2}$}
{G2}-instantons over twisted connected sums}

An important method to produce examples of compact  $7$--manifolds with holonomy
exactly $\rm G_2$ is the \emph{twisted connected sum} (TCS) construction 
\cite{CHNP1,CHNP2,Kovalev}, outlined in Section \ref{sec:Building-Fano}.
It consists of gluing a pair of asymptotically cylindrical (ACyl) Calabi--Yau $3$--folds obtained from certain  smooth projective $3$--folds called  \emph{building blocks}. A building block $(Z,\kd)$
 is given by a projective morphism
$f: Z\to \CP^1$ such that $\kd:=f^{-1}(\infty)$ is a smooth anticanonical  K3
surface, under some mild topological assumptions (see Definition \ref{def: building block}); in particular, $\kd$ has
trivial normal bundle. Choosing a convenient K\"ahler structure on
$Z$, one can make $V:=Z\setminus \kd$ into an ACyl Calabi--Yau $3$--fold, that is, a non-compact Calabi--Yau manifold with a tubular end modelled on $\mathbb{R}_+\times \mathbb{S}^1\times \kd$ \cite[Theorem 3.4]{CHNP2}. Then $\mathbb{S}^1\times V$ is an ACyl $\rm G_2$--manifold with a tubular end modelled
on $\mathbb{R}_+\times \mathbb{T}^2\times \kd$. 

When a pair $(Z_\pm,S_\pm)$ of building blocks  admits a \emph{matching} $\hkr:\kd_+ \to \kd_-$
(see Definition \ref{def:match}), there exists a so-called \hk rotation between the K3 surfaces `at infinity'. In this case, the corresponding pair $\mathbb{S}^1 \times V_\pm$ of  ACyl $\rm G_2$--manifolds is truncated at a
large `neck length' $T$ and, intertwining the circle components in the tori $\mathbb{T}^2_\pm$ along the tubular end, glued to form a compact $7$-manifold
$$
X =Z_+\#_\hkr Z_-:= S^1 \times V_+ \cup_{\hkr} S^1 \times V_-.
$$
For large enough $T_0$, this twisted connected sum $X$ carries a family of  $\rm G_2$-structures $\{\phi_T\}_{T\geq T_0}$ with $\Hol(\phi_T)=\rm G_2$
\cite[Theorem 3.12]{CHNP2}.
The construction is summarised in the following statement.

\begin{thm}[{\cite[Corollary 6.4]{CHNP2}}]
\label{thm:tcs}
Given a matching $\hkr : \kd_+ \to \kd_-$ between a pair
of building blocks $(Z_\pm, \kd_\pm)$ with K\"ahler classes
$\kclass_\pm \in H^{1,1}(Z_\pm)$ such that
$(\kclass_{+|\kd_+})^2 = (\kclass_{-|\kd_-})^2$, there exists
a family of torsion-free $\rm G_{2}$-structures
$\left\{\phi_{T}:T\gg1\right\}$ on the closed
$7$-manifold $X = Z_+\#_\hkr Z_-$ with $\Hol(\phi_T)=\rm G_2$.
\end{thm}

Theorem \ref{thm:tcs} raises a natural programme in gauge theory, aimed at
constructing $\rm G_2$-instantons over compact manifolds obtained as a TCS, originally
outlined in \cite{Henrique1}. Starting
from holomorphic bundles over $Z_\pm$ with a suitable
stability property, corresponding to Hermitian Yang-Mills metrics over the
ACyl Calabi-Yau components   \cite[Theorem 58]{Henrique2}, it is possible to glue 
a hypothetical pair of such solutions into a $\rm G_2$-instanton, provided a number of technical conditions
are met (see below).
 In the present paper we develop a constructive method to obtain explicit examples of such instanton gluing in many interesting cases, so it is important
to recall in detail the assumptions of this gluing theorem.

Let $A$ be an ASD instanton on a $\PP U(n)$-bundle $\mathscr{F}$ over a K\"ahler surface $\kd$. The linearisation of the instanton moduli space $\mathcal{M}_{S}$ near $A$ is modelled on the kernel of the deformation operator
$$
\mathbb{D_A}:= \dif^{*}_{A}\oplus\dif^{+}_{A}: \Omega^{1}(\kd,\mathfrak{g}_{\mathscr{F}})
\to 
(\Omega^{0}\oplus\Omega^{+})(\kd,\mathfrak{g}_{\mathscr{F}}),
$$ 
where $\mathfrak{g}_{\mathscr{F}}$ denotes the adjoint bundle associated to $\mathscr{F}$.  Let  $F$ be the corresponding holomorphic vector bundle
(\cf Donaldson-Kronheimer \cite{Donaldson}), and denote by $f$ the Hitchin-Kobayashi isomorphism: 
\begin{equation}
\label{eq: isomorphism f}
f: H^{1}(\kd,\Endr_{0}(F))\overset{\sim}{\longrightarrow} H^{1}_{A}:=\ker\mathbb{D_A}.
\end{equation}

\begin{thm}[{\cite[Theorem 1.2]{Henrique0}}]
\label{thm:HenriqueThomas}
Let $Z_{\pm}$ ,$\kd_{\pm}$, $\kclass_{\pm}$, $\hkr$, $X$ and $\phi_T$ be
as in
Theorem \ref{thm:tcs}.
Let $F_{\pm} \to Z_{\pm}$ be a pair of holomorphic vector bundles such
that the following hold: 
\begin{description}
\item[Asymptotic stability]

$F_{\pm}|_{\kd_{\pm}}$ is $\mu$-stable with respect to
$\kclass_{\pm}|_{\kd_{\pm}}$.
Denote the corresponding ASD instanton by $A_{\infty,\pm}$. 

\item[Compatibility]
There exists a bundle isomorphism
$\overline{\hkr}:F_{+}|_{\kd_{+}}\rightarrow F_{-}|_{\kd_{-}}$ covering the \hk
rotation $\hkr$ such that $\overline{\hkr}^{*} A_{\infty,-}=A_{\infty,+}$.

\item[Inelasticity]
There are no infinitesimal deformations of $F_{\pm}$ fixing the restriction to
$\kd_{\pm}$:
$$H^{1}(Z_{\pm},\Endr_{0}(F_{\pm})(-\kd_{\pm}))=0.$$

\item[Transversality]
If $\lambda_{\pm}:=f_\pm\circ\res_\pm :
H^{1}(Z_{\pm},\Endr_{0}(F_{\pm}))\rightarrow H^{1}_{A_{\infty,\pm}}$ denotes
the composition of restrictions to $\kd_{\pm}$ with the isomorphism
\eqref{eq: isomorphism f}, then the image of $\lambda_{+}$ and
$\overline{\hkr}^{*}\circ \lambda_{-}$ intersect trivially in
the linear space $H^{1}_{A_{\infty,+}}$:
$$\image (\lambda_{+})\cap \image (\overline{\hkr}^{*}\circ\lambda_{-})=\left\{0\right\}.$$
\end{description}
Then there exists a $U(r)$-bundle $\mathscr{F}$ over $X$ and a family of connections $\left\{A_{T}\ :\ T\gg1\right\}$ on the associated
$\PP U(r)$-bundle, such that each $A_{T}$ is an irreducible unobstructed
$G_{2}$-instanton over $(X,\phi_{T})$.
\end{thm}

Geometrically, the maps $\lambda_+$ and
$\overline{\hkr}^{*}\circ\lambda_{-}$ can be seen as linearisations of the
natural inclusions of the moduli of asymptotically stable bundles
$\mathcal{M}_{Z_\pm}$ into the moduli of ASD instantons $\mathcal{M}_{S_+}$
over the K3 surface `at infinity', and we think of $H^{1}_{A_{\infty,+}}$ as a
tangent model of  $\mathcal{M}_{S_+}$ near the ASD instanton $A_{\infty,+}$.
Then the transversality condition asks that the actual inclusions intersect
transversally at $A_{\infty,+}\in\mathcal{M}_{S_+}$.
That the intersection points are isolated reflects that the resulting
$\rm G_2$-instanton is rigid, since it is unobstructed and the deformation
problem has index 0.

\subsection{Gluing Hartshorne-Serre instanton bundles}

In \cite{CHNP1, CHNP2,Kovalev}, building blocks $Z$ are produced by blowing up
Fano or semi-Fano 3-folds along the base curve $\ccal$ of an anticanonical pencil
(see Proposition \ref{FanoBlock}). By
understanding the deformation theory of pairs $(Y,\kd)$ of semi-Fanos $Y$
and anticanonical K3 divisors $\kd \subset Y$, one can produce hundreds of thousands of pairs with the required matching (see Section \ref{subsec:match}).

In order to apply Theorem \ref{thm:HenriqueThomas} to produce $\rm G_2$-instantons over the
resulting twisted connected sums, one first requires some supply of
asymptotically stable, inelastic vector bundles $F \to Y$. Moreover, to satisfy
the hypotheses of compatibility and transversality, one would in general need
some understanding of the deformation theory of triples $(Y, \kd, F)$.

It is important to observe that in the so-called \emph{rigid} case, when
$H^{1}_{A_{\infty,+}}=\{0\}$, transversality is automatic, since the instantons that
are glued are isolated points in their moduli spaces. Using rigid
bundles adds further constraints to the matching problem for the building
blocks, but during the preparation of this article Walpuski \cite{Thomas} was able to exhibit one such example.

In this paper we study the non-rigid case, where the moduli spaces involved
are non-trivial, and transversality is a genuine condition.
Such examples have not previously appeared in the literature, but
are relevant because they open the possibility of
obtaining a conjectural instanton number on the $\rm G_2$-manifold $X$ as a genuine Lagrangian intersection
within the moduli space $\mathcal{M}_{S_+}$ over the K3 cross-section along the neck, which can be addressed
by enumerative methods in the future. 

Our method is to use the Hartshorne-Serre construction to obtain families of
bundles over the building blocks, for which the deformation theory can be
sufficiently explicitly understood to solve the matching problem and prove
transversality.
As a proof of concept, we focus on a special case where the moduli space
of asymptotically stable bundles over each building block is parametrised by
the blow-up curve $\ccal$ itself. This simplifies the problem by separating the
deformation theory of the bundles from the deformation theory of the pair
$(Y,\kd)$.
We can therefore \emph{first} find matchings between two semi-Fano families
using the techniques from \cite{CHNP2}, and \emph{then} exploit the high degree
of freedom in the choice of the blow-up curve $\ccal$ (see Lemma
\ref{lem:wiggle}) to satisfy the compatibility and transversality hypotheses.

We carry out all the computations for one particular pair of building blocks,
which is detailed in Examples \ref{ex:2conics} and \ref{ex:2conics2}.

\begin{thm}\label{thm:main}
There exists a matching pair of building blocks $(Z_\pm, \kd_\pm)$,
obtained as $Z_\pm=\Bl_{\ccal_\pm}Y_\pm$ for  $Y_+=\mathbb{P}^1\times\mathbb{P}^2$ and the double cover $Y_-\overset{2:1}{\longrightarrow}\mathbb{P}^1\times\mathbb{P}^2$ branched over
 a $(2,2)$ divisor, with rank $2$ holomorphic bundles $F_\pm \to Z_\pm$ satisfying the
hypotheses of Theorem \ref{thm:HenriqueThomas}.
\end{thm}

\subsection{Survey of the proof of Theorem \ref{thm:main}}

\begin{itemize}

\item
We construct holomorphic bundles on building blocks from certain complete intersection subschemes, via the Hartshorne-Serre correspondence [Theorem \ref{thm: Hartshorne-Serre}].
In Section \ref{general}, we establish conditions on the parameters of the
Hartshorne-Serre construction that are conducive to application of
Theorem \ref{thm:HenriqueThomas}. 
In Sections \ref{construction} and \ref{construction2}, we construct families 
of bundles $\{E_{\pm}\} $, over the particular blocks $Y_\pm$ of Theorem \ref{thm:main}, satisfying
these constraints.

\item
In 
Section \ref{subsec: stabi1S}, we recall sufficient conditions for the stability of $E_{\pm|\kd_{\pm}}$. Then, in Section \ref{sec moduli assoc to E|S}, we focus on the moduli space
$\mcal_{\kd_+,\mathcal{A}_+}^{s}(v_{\kd_+})$  of stable bundles on $\kd_+$, where the problems of compatibility and transversality therefore ``take place''.
Here $Y_+=\mathbb{P}^1\times\mathbb{P}^2$, 
$\kd_+ \subset Y_+$ is the anti-canonical K3 divisor and, for a smooth curve $\ccal_+ \in |{-}K_{Y_+|\kd_+}|$, the block
$Z_+ := \textrm{Bl}_{\ccal_+} Y_+$ is in the family obtained from Example \ref{ex:2conics}.

We show that $\mcal_{\kd_+,\mathcal{A}_+}^{s}(v_{\kd_+})$ is isomorphic to $\kd_+$ itself, and that the restrictions of the family of bundles $E_+$
correspond precisely to the blow-up curve $\ccal_+$.
Now, given a rank $2$ bundle $E_+\to Z_+$ such that $\mathcal{G}:= E_{+|\kd_{+}} \in \mcal^{s}_{\kd_{+},\mathcal{A}_{+}}(v_{\kd_+})$,
 the restriction map
\begin{equation}
\label{eq:resintro}
\res : H^1(Z_+, \Endr_0(E_+)) \to H^1(\kd_+, \Endr_0(\mathcal{G}))
\end{equation}
corresponds to the derivative at $E_+$ of the map between instanton moduli
spaces. Combining with Lemma \ref{lem:wiggle}, which guarantees the freedom
to choose $\ccal_+$ when constructing the block $Z_+$ from $\kd_+$, we arrive
at the following key step.

\begin{thm}
\label{thm:summaryintro}
For every
$\mathcal{G} \in  \mcal_{\kd_+,\mathcal{A}_+}^{s}(v_{\kd_+})$ and
every line $V \subset H^1(\kd_+, \Endr_0(\mathcal{G}))$, there is
a smooth base locus curve $\ccal_+ \in |{-}K_{Y_+|\kd_+}|$ and an exceptional fibre $\ell_+\subset\widetilde{\ccal}_+$ corresponding by Hartshorne-Serre to an  inelastic vector bundle $E_+ \to Z_+$, such that $E_{+|\kd_+} = \mathcal{G}$ and the  restriction map
\eqref{eq:resintro} has image $V$.
\end{thm}

\item
In Section \ref{sec:inelasticity} we give the rather technical proof that
the bundles $E_\pm$ are inelastic, together with some auxiliary
topological properties.

\item
Finally, in Section \ref{sec:wrap} we explain how to deduce Theorem
\ref{thm:main} from Theorem \ref{thm:summaryintro}. More precisely,
let $\hkr : \kd_+ \to \kd_-$ be a matching between $Y_+=\mathbb{P}^1\times\mathbb{P}^2$ and $Y_-\overset{2:1}{\longrightarrow}\mathbb{P}^1\times\mathbb{P}^2$.
Then Theorem \ref{thm:transv} argues that
for any $E_- \to Z_-$ as above we can (up to a twist by holomorphic line
bundles $\mathcal{R}_\pm\to Z_\pm$) choose the smooth curve
$\ccal_+ \in |{-}K_{Y_+ | \kd_+}|$ in the construction of $Z_+$ so that there
is a Hartshorne-Serre bundle $E_+\to Z_+$ that matches $E_-$ transversely.
Then the bundles $F_\pm:=E_\pm\otimes \mathcal{R}_\pm  $ satisfy
all the gluing hypotheses of Theorem \ref{thm:HenriqueThomas}.
\end{itemize}

While we made the expository choice of unfolding the construction of an example progressively along the paper, an alternative read focused on the general theory could follow through Sections \ref{sec:Building-Fano}, \ref{subsec:match}, \ref{general}, \ref{subsec: stabi1S}, \ref{sec: inelasticity of AS bundles} and \ref{sec: Inelasticity of H-S bundles}. 

\begin{rmk}
Theorem \ref{thm:summaryintro} is stronger than required by our argument. Indeed should the claim hold merely for \emph{generic} 
$\mathcal{G}$ and $V$, then we could argue that  every $E_- \to Z_-$ has a perturbation that can be matched transversely
by some $E_+ \to Z_+$, which is good enough to construct examples.
\end{rmk}

\subsection*{Acknowledgements}
\addcontentsline{toc}{subsection}{Acknowledgements}

We thank Daniele Faenzi, Marcos Jardim and Thomas Walpuski for many important
discussions. 
In particular, we acknowledge Marcos Jardim for suggesting the Hartshorne-Serre technique to produce bundles parametrised by curves.
GM is supported by S\~ao Paulo Research Foundation (Fapesp) grant 2014/05733-9 and the Marco Brunella Grant of Burgundy University. HSE is supported by Fapesp
grant 2014/24727-0 and the Brazilian National Council for Scientific and Technological Development (CNPq) Productivity Grant 312390/2014-9.
JN is supported by the Simons Foundation under the Simons Collaboration
on Special Holonomy in Geometry, Analysis and Physics
(grant \#488631, Johannes Nordström).

\newpage
\section{\texorpdfstring{$\rm G_{2}$}
{G2}-manifolds via semi-Fano 3-folds}

The goal of this section is to present a concrete example of a matching of
building blocks, constructed from a certain pair of Fano 3-folds.
We begin by reviewing some background about the construction and the matching
problem.

\subsection{Building blocks from semi-Fano 3-folds and twisted connected sums}
\label{sec:Building-Fano}

\begin{defi}
\label{def: building block}
A \emph{building block} is a nonsingular algebraic 3-fold $Z$ together with a projective morphism $f: Z\rightarrow \CP^{1}$ satisfying the following assumptions:
\begin{enumerate}
\item
The anti-canonical class $-K_{Z}\in H^{2}(Z,\Z)$ is primitive.
\item
$\kd=f^{-1}(\infty)$ is a non-singular K3 surface and $\kd\sim -K_{Z}$.

Identify $H^{2}(\kd,\Z)$ with the K3 lattice $L$ (\ie choose a marking for $\kd$), and let $N$ denote the image of $H^{2}(Z,\Z)\rightarrow H^{2}(\kd,\Z)$.
\item
The inclusion $N\hookrightarrow L$ is primitive.
\item
The groups $H^{3}(Z,\Z)$ and $H^{4}(Z,\Z)$ are torsion-free.
\end{enumerate}
\end{defi}

In particular, building blocks are simply-connected \cite[\S 5.1]{CHNP1}. Theorem \ref{thm:tcs} states that one can construct closed \gtmfd s from
pairs of building blocks that match in the following sense.

\begin{defi}
\label{def:match}
Let $Z_\pm$ be complex 3-folds, $\kd_\pm \subset Z_\pm$ smooth anticanonical
K3 divisors and $\kclass_\pm \in H^2(Z_\pm)$ K\"ahler classes.
We call a \emph{matching} of $(Z_+, \kd_+, \kclass_+)$
and $(Z_-, \kd_-, \kclass_-)$ a diffeomorphism $\hkr \colon {\kd_+ \to \kd_-}$ such that
$\hkr^* \kclass_- \in H^2(\kd_+)$ and $(\hkr^{-1})^* \kclass_+ \in H^2(\kd_-)$
have type $(2,0) + (0,2)$.

We also say that $\hkr : \kd_+ \to \kd_-$ is a matching of
$Z_+$ and $Z_-$ if there are K\"ahler classes $\kclass_\pm$ so that the above
holds.
\end{defi}

Let us briefly summarise the construction in Theorem \ref{thm:tcs}.
For any building block $(Z,\kd)$, the noncompact $3$--fold  
$V := Z \setminus \kd$ admits ACyl Ricci-flat K\"ahler metrics
\cite[Theorem D]{HHN}, hence
an ACyl Calabi-Yau structure. This Calabi-Yau structure can be specified
by choosing a K\"ahler class $\kclass \in H^{1,1}(Z)$ and a meromorphic
$(3,0)$-form with a simple pole along $\kd$. The asymptotic limit of the
Calabi-Yau structure defines a \hk structure on $\kd$.

Given a pair of such Calabi-Yau manifolds $V_\pm$ and a so-called
\emph{\hk rotation} $\hkr : \kd_+ \to \kd_-$ (see \cite[Definition 3.9]{CHNP2}), one can apply \cite[Theorem 5.34]{Kovalev} to
glue $S^1 \times V_\pm$ into a closed manifold $X$ with a $1$-parameter family of torsion-free \gtstr s
(see  \cite[Theorem 3.12]{CHNP2}).
Given a matching $\hkr$ between a pair of
building blocks $(Z_\pm, \kd_\pm, \kclass_\pm)$, 
one can make the choices in the definition of the ACyl Calabi-Yau structure
so that $\hkr$ becomes a \hk rotation of the induced \hk structures
(\cf \cite[Theorem~3.4 and Proposition 6.2]{CHNP2}).
Combining these steps proves Theorem \ref{thm:tcs}.

For all but 2 of the 105 families of Fano $3$-folds, the base locus of a generic
anti-canonical pencil is smooth. This also holds for most families in the
wider class of `semi-Fano $3$-folds' in the terminology of \cite{CHNP1},
\ie smooth projective $3$-folds where $-K_Y$ defines a morphism that does not
contract any divisors. 
We can then obtain building blocks using \cite[Proposition 3.15]{CHNP2}:

\begin{prop}\label{FanoBlock}
Let $Y$ be a semi-Fano 3-fold with $H^{3}(Y,\Z)$ torsion-free, $|\kd_{0},\kd_{\infty}|\subset |-K_{Y}|$ a generic pencil with (smooth) base locus $\ccal$, $\kd\in |\kd_{0},\kd_{\infty}|$ generic, and $Z$ the blow-up of $Y$ at $\mathscr{C}$. Then $\kd$ is a smooth K3 surface, its proper transform in $Z$ is isomorphic to $\kd$, and $(Z,\kd)$ is a building block. Furthermore
\begin{enumerate}
\item
the image $N$ of $H^{2}(Z,\Z)\rightarrow H^{2}(\kd,\Z)$ equals that of $H^{2}(Y,\Z)\rightarrow H^{2}(\kd,\Z)$;

 \item
 $H^{2}(Y,\Z)\rightarrow H^{2}(\kd,\Z)$ is injective and the image $N$ is primitive in $H^{2}(\kd,\Z)$.
\end{enumerate}
\end{prop}

\begin{rmk}
Alternatively we could say that $\kd \in |-K_Y|$ and $\ccal \in |-K_{Y|S}|$
are generic and smooth. For $H^0(Y, -K_Y) \to H^0(\kd, -K_{Y|\kd})$
is surjective, so there really is an $\kd_\infty \in |{-}K_Y|$ that intersects
$\kd$ in~$\ccal$, and then $|\kd_\infty : \kd|$ is a pencil with base locus
$\ccal$.
\end{rmk}

Note that if $Y_\pm$ is a pair of semi-Fanos and $\hkr : \kd_+ \to \kd_-$
is a matching in the sense of Definition \ref{def:match}, then $\hkr$ also
defines a matching of building blocks constructed from $Y_\pm$ using 
Proposition \ref{FanoBlock}. Thus given a pair of matching semi-Fanos we can
apply Theorem \ref{thm:tcs} to construct closed \gtmfd s, \emph{but} this
still involves choosing the blow-up curves $\ccal_\pm$.
For later use 
we make an observation concerning these blow-up curves, which will play an especially important role in our transversality argument
in Section \ref{sec moduli assoc to E|S}.

\begin{lemme}
\label{lem:wiggle}
Let $Y$ be a semi-Fano, $\kd \in |{-}K_Y|$ a smooth K3 divisor, and suppose
that the restriction of $-K_Y$ to $\kd$ is very ample.
Then given any point $x \in \kd$ and any
(complex) line $V \subset T_x \kd$, there exists an anticanonical pencil
containing $\kd$ whose base locus $\ccal$ is smooth, contains $x$, and
$T_x \ccal = V$.
\end{lemme}

\begin{proof}
The sections of $-K_{Y|\kd}$ define an embedding
$\kd \hookrightarrow \CP^g$, for some $g \geq 3$.
The image of $V$ defines a line
in~$\CP^g$, intersecting $\kd$ in a finite number of points (generically
just in~$x$ if $g > 3$). Consider the sections of $\kd$ by hyperplanes
$H \subset \CP^g$ that contain $V$.
These form a $(g{-}2)$--dimensional family, with base locus $\kd \cap V$. By
Bertini's theorem, a generic section $H \cap \kd$ in this family is smooth away
from $\kd \cap V$.  On the other hand, for each point $y \in \kd \cap V$,
certainly a generic section is smooth at $y$---indeed, $H \cap \kd$ is smooth
at $y$ as long as $T_y H$ does not contain $T_y \kd$.
Hence there is a smooth section $\ccal := H \cap \kd$ with $T_x \ccal = V$.
\end{proof}

\subsection{The matching problem}
\label{subsec:match}

We now explain in more detail the argument of \cite[\S 6]{CHNP2} for finding matching building blocks
$(Z_\pm, \kd_\pm)$.
The blocks will be obtained by applying Proposition \ref{FanoBlock} to a pair
of semi-Fanos $Y_\pm$, from some given pair of deformation types $\sff_\pm$.

A key deformation invariant of a semi-Fano $Y$ is its Picard lattice
$\Pic(Y) \cong H^2(Y; \Z)$. For any anticanonical K3 divisor
$\kd \subset Y$, the injection $\Pic(Y) \into H^2(\kd;\Z)$ is primitive.
The intersection form on $H^2(\kd;\Z)$ of any K3 surface is isometric to
$L_{K3} := 3U \oplus 2E_8$, the unique even unimodular lattice of
signature $(3,19)$.
We can therefore identify $\Pic(Y)$ with a primitive sublattice
$N \subset L_{K3}$ of the K3 lattice, uniquely up to the action of the isometry
group $O(L_{K3})$ (this is usually uniquely determined  by the isometry
class of $N$ as an abstract lattice).

Given a matching $\hkr \colon \kd_+ \to \kd_-$ between anticanonical
divisors in a pair of semi-Fanos, we can choose the isomorphisms $H^2(\kd_\pm;\Z) \cong L_{K3}$ compatible with $\hkr^*$, 
hence identify $\Pic(Y_+)$ and $\Pic(Y_-)$
with a \emph{pair} of primitive sublattices $N_+, N_- \subset L_{K3}$.
While the $O(L_{K3})$ class of $N_\pm$ individually depends only on $Y_\pm$,
the $O(L_{K3})$ class of the pair $(N_+, N_-)$ depends on $\hkr$, and we
call $(N_+, N_-)$ the \emph{configuration} of $\hkr$.

Many important properties of the resulting twisted connected sum only depend on the
\hk rotation in terms of  the configuration. Given a pair $\sff_\pm$ of
deformation types of semi-Fanos it is therefore interesting to know which
configurations of their Picard lattices are realised by some \hk rotation.
Let us use the following terminology. Given a primitive sublattice
$N \subset L_{K3}$ and $A \in N$ such that $A^2 > 0$, recall that an
\emph{$N$-polarised} K3 surface is a K3 surface $\kd$ together with a marking
$h \colon H^2(\kd;\Z) \cong L_{K3}$ such that $N \subseteq h(\Pic(\kd))$,
and $A$ corresponds to an ample class on $\kd$.
For an open subcone $\amp_\sff \subset N \otimes \R$, let us call a set
$\sff$ of semi-Fanos \emph{$(N, \amp_\sff)$-generic} if
a generic $N$-polarised K3-surface $(\kd,h)$ has an embedding
$i \colon \kd \into Y$ as an anticanonical divisor in some $Y \in \sff$,
such that $h \circ i^* : \Pic(Y) \to L_{K3}$ is an isomorphism onto $N$
and the image of the ample cone of $Y$ contains $\amp_\sff$.

Given a configuration $N_+, N_- \subset L_{K3}$, let
\[
N_0 := N_+ \cap N_- , 
\quad\text{and}\quad
\nres_\pm := N_\pm \cap N_\mp^\perp .
\]
We say that the configuration is \emph{orthogonal} if $N_\pm$ are rationally
spanned by $N_0$ and $\nres_\pm$
(geometrically, this means that the reflections in $N_+$ and $N_-$ commute).
Then there are sufficient conditions for a given orthogonal configuration
to be realised by some matching \cite[Proposition 6.17]{CHNP2}:  
\begin{prop}
\label{6.18}
Let $N_{\pm} \subset L_{K3}$ be a
configuration of two primitive sub-lattices of signatures $(1,r_{\pm}{-}1)$.
Let $\mathcal{Y}_{\pm}$ be $(N_\pm, \amp_{\sff_\pm})$-generic sets of 
semi-Fano 3--folds, and assume that
\begin{itemize}
\item the configuration is orthogonal, and
\item $R_\pm \cap \amp_{\sff_\mp} \not= \emptyset$.
\end{itemize} 
Then there exist $Y_{\pm}\in \mathcal{Y}_{\pm}$,
$\,\kd_{\pm}\in |{-}K_{Y_{\pm}}|$, and a matching
$\hkr:\kd_{+}\rightarrow \kd_{-}$ with the given configuration.
Moreover, the K\"ahler classes $\kclass_\pm$ on $Y_\pm$ in
Definition \ref{def:match} can be chosen so that $\kclass_{\pm|\kd_\pm}$
is arbitrarily close to any given element $R_\pm \cap \amp_{\sff_\mp}$. 
\end{prop}

Fixing henceforth a primitive sublattice
$N \subset L_{K3}$,  every
nonempty deformation type $\sff$ of semi-Fano $3$-folds is $(N, \amp_\sff)$-generic
for some $\amp_\sff$ \cite[Proposition 6.9]{CHNP1} . For most pairs of deformation types $\sff_\pm$, one can
apply results of Nikulin \cite{Lattice} to embed the perpendicular direct sum $N_+ \perp N_-$
primitively in $L_{K3}$.
Thus one obtains a configuration satisfying the hypotheses
of Proposition~\ref{6.18}. This is used in \cite{CHNP2} and \cite{Crowley}
to produce many examples of twisted connected sum \gtmfd s.

Now consider the problem of finding matching bundles $F_\pm \to Z_\pm$ in order
to construct $\rm G_2$-instantons by application of Theorem \ref{thm:HenriqueThomas}.
For the compatibility hypothesis it is necessary that
\[ c_{1}(F_{+}|_{\kd_{+}}) = \hkr^{*}c_{1}(F_{-}|_{\kd_{-}}) \in H^2(\kd_+) . \]
Identifying $H^2(\kd_+; \Z) \cong L_{K3} \cong H^2(\kd_-; \Z)$ compatibly with
$\hkr^*$, this means we need
\[ c_1(F_{+|\kd_+}) \; = \; c_1(F_{-|\kd_-}) \; \in \; N_+ \cap N_-
\; = \; N_0 . \]
Hence, if $N_0$ is trivial,  both $c_1(F_{\pm|\kd_\pm})$ 
must also be trivial, which is a very restrictive condition on the first Chern classes of our bundles.
To allow more possibilities, we want  matchings $\hkr$
whose configuration $N_+, N_- \subset L_{K3}$ has non-trivial intersection
$N_0$.

Table 4 of \cite{Crowley} lists all 19 possible matchings
of Fano 3-folds with Picard rank 2 such that $N_0$ is non-trivial.
In this paper, the pair of building blocks we will consider comes from that
table. The relevant Fano 3-folds are $Y_+=\mathbb{P}^1\times\mathbb{P}^2$ and the double cover  $Y_-\overset{2:1}{\longrightarrow}\CP^1\times\CP^2$ branched
over a $(2,2)$ divisor.
We explain the reasons why these Fano 3-folds were chosen in Remark
\ref{ChoiceY}. Moreover, we suggest other Fano 3-folds that are relevant  to produce examples of $\rm G_2$-instantons.

\subsection{The Fano 3-folds 
\texorpdfstring{$Y_+=\CP^1\times\CP^2$}
{Y+ = P1 x P2}
and 
\texorpdfstring{$Y_-\protect\overset{2:1}{\longrightarrow}\CP^1\times\CP^2$}
{Y-  -> P1 x P2}
}
\label{sec:semi-Fano Blow-up of P3}

\begin{ex}
\label{ex:2conics}
The product $Y_+=\mathbb{P}^1\times\mathbb{P}^2$ is a Fano 3-fold. 
Let $\left|\kd_{0},\kd_{\infty}\right|\subset \left|-K_{Y_+}\right|$ be a generic
pencil with (smooth) base locus $\ccal_+$ and
$\kd_+\in \left|\kd_{0},\kd_{\infty}\right|$ generic.
Denote by $r_+:Z_+\rightarrow Y_+$  the blow-up of $Y_+$ in $\ccal_+$,  by 
 $\wt{\ccal_+}$ 
the exceptional divisor and by $\ell_+$ a fibre
of $p_{1}:\wt{\ccal_+}\rightarrow \mathscr{C}_+$. The proper transform of $\kd_+$ in $Z_+$ is also denoted by $\kd_+$, and
$(Z_+,S_+)$ is a building block by Proposition \ref{FanoBlock}.

We fix classes 
$$
H_+:=r^{*}_+(\left[\mathbb{P}^1\times\mathbb{P}^1\right])
\qand
G_+=r^{*}_+(\left[\left\{x\right\}\times\mathbb{P}^2\right]) \in H^2(Z_+),
$$
where $x$ is a point, and also
$$
h_+:=r^*_+(\left[\left\{x\right\}\times\CP^1\right])
\qand
g_+:=r^*_+(\left[\CP^1\times\left\{x\right\}\right]) \in H^4(Z_+).
$$  The Picard group of $\kd_+$
has rank at least $2$, containing 
$$
A_+:=G_{+|S_+}
\qand
B_+:=H_{+|S_+}.
$$
Moreover, the sublattice $N_+$ spanned by $A_+$ and $B_+$ has intersection
form represented by the matrix $$M_+:=
\begin{pmatrix}
0 &3    \\
3 & 2 \\
 \end{pmatrix}.$$

\noindent \textbf{NB.:} Clearly $-K_{Y_+}$ is very ample, thus also  $-K_{Y_+|\kd_+}$, so $Y_+$ lends itself to application of Lemma \ref{lem:wiggle}. 
\end{ex}

\begin{ex}
\label{ex:2conics2}
A double cover $\pi:Y_-\overset{2:1}{\longrightarrow}
\mathbb{P}^1\times\mathbb{P}^2$ branched over a smooth $(2,2)$ divisor $D$ is
a Fano 3-fold.
Let $\left|\kd_{0},\kd_{\infty}\right|\subset \left|-K_{Y_-}\right|$ be a
generic pencil with (smooth) base locus $\ccal_-$ and
$\kd_-\in \left|\kd_{0},\kd_{\infty}\right|$ generic.
Denote by $r_-:Z_-\rightarrow Y_-$ the blow-up of $Y_-$ in $\ccal_-$, and by  $\wt{\ccal_-}$ the exceptional divisor.
The proper transform of $\kd_-$ in $Z_-$ is also denoted by $\kd_-$, and
$(Z_-,S_-)$ is a building block by Proposition \ref{FanoBlock}.
We fix classes
\[
H_-:=(r_-\circ \pi)^{*}(\left[\mathbb{P}^1\times\mathbb{P}^1\right])
\qand
G_-=(r_-\circ \pi)^*(\left[\left\{x\right\}\times\mathbb{P}^2\right])
\in H^2(Z_-),
\]
where $x$ is a point, and also
\[
h_-:=\frac{1}{2}(r_-\circ \pi)^*(\left[\left\{x\right\}\times\CP^1\right])
\qand
g_-:=\frac{1}{2}(r_-\circ \pi)^*(\left[\CP^1\times\left\{x\right\}\right])
\in H^4(Z_-).
\]
For a generic $x \in \CP^2$, the curve $\{x\} \times \CP^1$ meets the branching
divisor $D$ transversely in two points, and the pre-image of
$\{x\} \times \CP^1$ in $Y_-$ is an irreducible rational curve, whose
Poincar\'e dual is mapped to $2h_-$ by $r_-^*$.
Note, however, that there is a quartic curve $Q \subset \CP^2$
(defined by the discriminant of the quadric polynomial corresponding to
restriction of $D$ to $\{x\} \times \CP^1$) such that for generic $x \in Q$,
the curve $\{x\} \times \CP^1$ is tangent to~$D$. For such $x$, the pre-image
of $\{x\} \times \CP^1$ in $Y_-$ is a union of two lines. Such lines
are parametrised by the pre-image $\widetilde Q$ of $Q$ in $\kd_-$. The proper
transform $W$ in $Z_-$ of such a line is Poincar\'e dual to $h_-$.

The Picard group of $\kd_-$ has rank at least $2$, containing 
\[
A_-:=G_{-|S_-}
\qand
B_-:=H_{-|S_-}.\]
\vspace{-1.5\baselineskip}

\noindent
The sublattice $N_-$ generated by these vectors has intersection form represented by $M_-:=
\begin{pmatrix}
0 &4    \\
4 & 2 \\
 \end{pmatrix}$.
\end{ex}

According to Table 5 of \cite{Crowley}, we can find a matching between $Y_+$ and $Y_-$ choosing the ample classes
$$\mathcal{A}_+:=A_++B_+  \qand \mathcal{A}_-:=2A_-+B_-.$$
Moreover $N_0\subset N_+$ is generated by $5A_+-3B_+$ and $N_0\subset N_-$ is generated by $5A_- - 2B_-$ (both have square $-72$).
 
Given a matching,
we can take any smooth
$\ccal_\pm \in |{-}K_{Y_\pm | \kd_\pm}|$ and apply Proposition \ref{FanoBlock}
to construct building blocks $(Z_\pm,\kd_\pm)$, then apply Theorem \ref{thm:tcs} to obtain a twisted connected sum.

\section{Twisted Hartshorne-Serre bundles over building blocks}

The Hartshorne-Serre construction generalises the correspondence
between divisors and line bundles, under certain conditions, in the sense that bundles of higher
rank are associated to subschemes of higher codimension. We recall
the rank $2$ version, as an instance of Arrondo's formulation \cite[Theorem 1]{Arrondo}:
\begin{thm}\label{thm: Hartshorne-Serre}
Let $W\subset Z$ be a local complete intersection subscheme of codimension 2 in a smooth algebraic variety. If there exists a line bundle $\mathcal{L}\to Z$ such that
\begin{itemize}
        \item
        $H^2(Z,\mathcal{L}^*)=0$,
        
        \item
        $\wedge^2\mathcal{N}_{W/Z}=\mathcal{L}|_W$,
\end{itemize}
then there exists a rank $2$ vector bundle 
such that
\begin{enumerate}
        \item
        $\wedge^2E=\mathcal{L}$,
        \item
        $E$ has a global section whose vanishing locus is $W$.
\end{enumerate}
\end{thm}

We will refer to such  $E$ as the \emph{Hartshorne-Serre bundle obtained from $W$ (and $\mathcal{L}$)}.

\subsection{A technique to construct matching bundles}\label{general}

Let $Y$ be a semi-Fano $3$--fold and $(Z,\kd)$ be the block constructed as a
blow-up of $Y$ along the base locus $\ccal$ of a generic anti-canonical pencil
[Proposition \ref{FanoBlock}].  
We now describe a general approach for making the choices of $\mathcal{L}$ and
$W$ in Theorem \ref{thm: Hartshorne-Serre}, in order to construct a Hartshorne-Serre bundle $E\to Z$ which, up to a twist, yields the bundle $F$
meeting the requirements of Theorem \ref{thm:HenriqueThomas}.

\begin{enumerate}
 \item\label{i}

As explained in Section \ref{subsec:match}, for compatibility we need $ c_{1}(F_{|\kd})\in N_{0}. $
However, this  is too restrictive for producing  bundles with suitable asymptotics under the
Hartshorne-Serre construction. Instead, we obtain a rank $2$ vector bundle $E$ and
a line bundle $\mathcal{R}$ such that:
\begin{equation}
c_{1}(E\otimes\mathcal{R}_{|\kd})\in N_{0}
\label{chernclass1}
\end{equation}
and set $F:=E\otimes\mathcal{R}$. The properties of asymptotic
stability and inelasticity will be equivalent for $E$ and for $F$, hence we can
work directly with $E$.
Moreover, since 
$$
c_{1}(E\otimes\mathcal{R})=c_{1}(E)+2c_{1}(\mathcal{R}),
$$
the existence of a line bundle $\mathcal{R}$ such that
\eqref{chernclass1} holds is equivalent to
$$
c_{1}(E_{|S}) \, \in \, N_{0} \, \md \, 2\Pic (\kd).
$$

\item\label{ii}
By the Hoppe stability criterion [Proposition \ref{Hartshorne}], if our $E$ is asymptotically stable with respect to a polarisation $\mathcal{A}\perp N_0$ in the K\"ahler cone $\mathcal{K}_{Z}$, then necessarily $\mu_{\mathcal{A}}(E_{|\kd})>0$, so one must also arrange
$$
c_{1}(E_{|\kd})\cdot \mathcal{A}>0.
$$

\item\label{iii}
If we choose a genus $0$ curve $W$ by identifying the first Chern classes, 
the condition 
$\wedge^2\mathcal{N}_{W/Z}=\mathcal{L}|_{W}$
of Theorem \ref{thm: Hartshorne-Serre} is equivalent to:
\begin{equation}
\label{SWL}
(S-c_1(\mathcal{L}))\cdot W=2.
\end{equation}

On the block $(Z_+,S_+)$, we choose a fibre  $W_+:=\ell_+$ of the map $p_{1}:\wt{\ccal}_+\rightarrow \mathscr{C}_+$, where $\wt{\mathscr{C}}_+$ 
is the exceptional divisor of $Z_+\rightarrow Y_+$, to obtain in fact a \emph{family} of bundles $\{E_+\to Z_+\}$ parametrised by $\ccal_+$. 
The large freedom to move the curve $\ccal_+$ without changing $\kd_+$,
as stated in Lemma \ref{lem:wiggle}, will be essential to the proof of
Theorem \ref{thm:summaryintro}. In this case,  (\ref{SWL}) becomes
$$
c_{1}(\mathcal{L}_+)\cdot \ell_+=-1.
$$
A little more generally, one could chose  $W_+$ as the disjoint union of $k$  fibres $\ell_1,...,\ell_k$ of $p_{1}$, provided  $\wedge^2\mathcal{N}_{\ell_i/Z_+}=\mathcal{L}|_{\ell_i}$ for each fibre. In any case, condition (\ref{SWL}) remains as above. 

\item\label{iiii}
We denote by $\mcal^{s}_{\kd,\mathcal{A}}(v)$ the moduli space of $\mathcal{A}$-$\mu$-stable bundles on $\kd$ with Mukai vector $v:=v(E_{|\kd})$ [\cf Section \ref{sec moduli assoc to E|S}]. 
According to Theorem \ref{thm: Hartshorne-Serre}, we have $c_{2}(E)=\left[W\right]$,
hence 
\begin{align}\label{dimk}
\begin{split}
        \dim \mcal^{s}_{\kd,\mathcal{A}}(v)
        &= 4\cdot c_{2}(E_{|\kd})-c_{1}(E_{|\kd})^{2}-6\\
        &= 4\cdot \kd\cdot W-c_{1}(E_{|\kd})^{2}-6.
\end{split}
\end{align}
In general, we may achieve $\dim \mcal^{s}_{\kd,\mathcal{A}}(v)=2k$, for $k\in\mathbb{N}^*$,  by adopting  $W_+=\sqcup\ell_i$ the disjoint union of exceptional fibres as above in \ref{iii}, so  (\ref{dimk}) reads 
$$c_{1}(E_{+|\kd_+})^{2}=2k-6.$$
However, to prove transversality, it will be simplest to impose 
$\dim \mcal^{s}_{\kd,\mathcal{A}}(v)=2$, so  (\ref{dimk}) becomes 
$2=\kd\cdot W-\frac{1}{4}c_{1}(E_{|\kd})^{2}$, and, choosing $W_+=\ell_+$, 
$$c_{1}(E_{+|\kd_+})^{2}=-4.$$
\item
Our most restrictive constraint is the vanishing of the Hartshorne-Serre obstruction [\cf  Theorem \ref{thm: Hartshorne-Serre}]:
\begin{equation}
H^2(Z,\mathcal{L}_+^*)=0, 
\label{HS2}
\end{equation}

Moreover,  condition \ref{iii} imposes $\mathcal{L}_+=-\kd_++r_+^*(D)$, with $D\in \Pic (Y_+)$, fitting in  the following exact sequence:
\begin{equation}
\xymatrix{ 
0\ar[r]&\mathcal{O}_{Z_+}(-r_+^*(D))\ar[r] & \mathcal{L}_+^*\ar[r] & \mathcal{O}_{Z_+}(-r_+^*(D))_{|\kd_+}\ar[r]& 0.
}
\label{geneexactsequence}
\end{equation}
Applying Serre duality to the associated long exact sequence, we have:
$$
\xymatrix{ 
H^2(Z,\mathcal{L}_+^*)\ar[r] & H^0(\kd_+,\mathcal{O}_{\kd_+}(D_{|\kd_+}))\ar[r]& H^0(Z,\mathcal{O}_{Z_+}(r_+^*(D)-\kd_+)).
}
$$
Theoretically, one could set out to prove that $H^0(\kd_+,\mathcal{O}_{\kd_+}(D_{|\kd_+}))\hookrightarrow H^0(Z,\mathcal{O}_{Z_+}(r_+^*(D)-\kd_+))$, but this is very unlikely,  because in non-trivial instances $r_+^*(D)-\kd_+$ is seldom effective. So in practice  we show that 
\begin{equation}
H^0(\kd_+,\mathcal{O}_{\kd_+}(D_{|\kd_+}))=0. 
\label{H0DS}
\end{equation}

Hence in general our technique leads to $D_{|\kd_+}$ which is not effective. One the other hand, by  \ref{ii}, $\mu_{\mathcal{A}_+}(D_{|\kd_+})>0$.
If we choose  $Y_+$ a rank $2$ semi-Fano and assume  \ref{i}, in practice, it has led to $D=H_+-G_+$, for some $H_+$ and $G_+$ generating $\Pic \kd_+$ (see \eg Section \ref{sec:semi-Fano Blow-up of P3}) such that $\mu_{\mathcal{A}_+}(H_{+|\kd_+})>\mu_{\mathcal{A}_+}(G_{+|\kd_+})$.

Suppose  that indeed $H^2(Z,\mathcal{L}^*)=0$ and $D=H_+-G_+$ as above.
Our technique to obtain  transversality imposes in addition  $H^1(Z_+,\mathcal{L}^*_+)=0$ (see Section \ref{sec moduli assoc to E|S}), in order to identify the bundles $E_{+|\kd_+}$ with points of $\kd_+$. 
Then the exact sequence (\ref{geneexactsequence}) yields:
$$H^1(\kd_+,\mathcal{O}_{\kd_+}(D_{|\kd_+}))=H^2(Z_+,\mathcal{O}_{Z_+}(-r_+^*(D))).$$
This condition is not so easy to check. However, when $H_+$ and $G_+$ are the classes of simply connected divisors, we can prove that the right-hand side vanishes using Lemma \ref{lemmeBasic} below;
together with  (\ref{H0DS}) this gives $\chi(D_{|\kd_+})=0$, which in turn implies $\chi(\mathcal{L}_+)=0$. 
Furthermore, by Riemann--Roch, $D_{|\kd_+}^2=-4$. This imposes $k=1$ in \ref{iiii}.

In conclusion, the conditions $H^2(Z,\mathcal{L}_+^*)=H^1(Z,\mathcal{L}_+^*)=0$ are quite restrictive and in trying to achieve them we may impose, by excess, $$
D_{|\kd_+}^2=-4 \qand \chi(\mathcal{L}_+)=0.
$$
\item
As to inelasticity, in the case   
$\dim \mcal^{s}_{\kd,\mathcal{A}}(v)=2$, Corollary \ref{CInelastic2} gives us the necessary and sufficient condition on the dimension of the family of curves of class $W$:
\begin{equation}\label{eq: dim 2 inelast}
  \dim H^0(W,\mathcal{N}_{W/Z})=\dim H^0(Z,E)+H^1(Z,\mathcal{L}^*),
\end{equation}
which further constrains the coupled choice of $W_\pm$ and $\mathcal{L}_\pm$.
Following  (v), we must actually  check that
$$
\dim H^0(W,\mathcal{N}_{W/Z})=\dim H^0(Z,E)=1+H^0(Z,\mathcal{L}\otimes \mathcal{I}_{W}),
$$
which can be calculated easily. 
\end{enumerate}

\begin{summ}\label{generaltec}
Let $(Z_\pm,\kd_\pm)$ be
the building blocks constructed by blowing-up  $N_\pm$-polarised semi-Fano $3$-folds $Y_\pm$ along the base locus $\ccal_\pm$ of
a generic anti-canonical pencil [\cf Proposition \ref{FanoBlock}]. 
Let $N_{0}\subset N_\pm$ be the sub-lattice of  orthogonal matching, as in Section \ref{subsec:match}.
Let $\mathcal{A}_\pm$ be the restriction of an ample class of $Y_\pm$ to $\kd_\pm$ which is orthogonal to $N_{0}$.  We look for the Hartshorne-Serre parameters $W_\pm$ and $\mathcal{L}_\pm$ of Theorem \ref{thm: Hartshorne-Serre}, where  $W_+=\ell_+$ is an exceptional fibre in $Z_+$,  $W_-$ is a genus $0$ curve in $Z_-$ and  $\mathcal{L}_\pm\to Z_\pm$ are line bundles such that:

\setlength{\columnsep}{-80pt}
\raggedcolumns
\begin{multicols}{2}
\begin{enumerate}
\item
$c_{1}(\mathcal{L}_{\pm|\kd_\pm}) \in  N_{0}  \;\md\; 2\Pic (\kd_\pm)$;

\item
$c_{1}(\mathcal{L}_{\pm|\kd_\pm})\cdot \mathcal{A}_\pm>0$;

\item
$c_{1}(\mathcal{L}_+)\cdot \ell_+=-1$ and \\ $(S_--c_1(\mathcal{L}_-))\cdot W_-=2$;

\item 
$c_{1}(\mathcal{L}_{+|\kd_+})^{2}=-4$ and $\kd_-\cdot W_- -\frac{1}{4}c_{1}(\mathcal{L}_{-|\kd_-})^{2}=2$;

\item
$\chi(\mathcal{L}^*_{+})=0$;

\item
$\dim H^0(W,\mathcal{N}_{W/Z})=1+H^0(Z,\mathcal{L}\otimes \mathcal{I}_{W})$.
\end{enumerate}
\end{multicols}

\end{summ}

\begin{rmk}\label{RemarkImportante}
Suppose  $F_{\pm}$ satisfy the hypotheses of Theorem
\ref{thm:HenriqueThomas}. Then the restrictions  $F_{\pm}|_{\kd_{\pm}}$
have degree $c_1(F_{\pm}|_{\kd_{\pm}})\cdot \mathcal{A}_\pm=0$, because
$c_1(F_{\pm}|_{\kd_{\pm}})\in N_0$ and $\mathcal{A}_\pm\bot N_0$.
Moreover, $F_{\pm}|_{\kd_{\pm}}$ are $\mu$-$\mathcal{A}_\pm$-stable,  hence also their duals, so $H^0(F_{\pm}|_{\kd_{\pm}})=H^0(F_{\pm}^*|_{\kd_{\pm}})=0$.
By Serre duality, this ensures that 
$H^2(F_{\pm}|_{\kd_{\pm}})=0$, thus
$$\chi(F_{\pm}|_{\kd_{\pm}})\leq0.$$
Furthermore, in order to get $2$-dimensional   moduli spaces $\mcal_\pm:=\mcal^{s}_{\kd_\pm,\mathcal{A}_\pm}(v(F_{\pm}|_{\kd_{\pm}}))$ from the formula (see Theorem \ref{thmMaruyama})
$$\dim \mcal_\pm=10-4\chi(F_{\pm}|_{\kd_{\pm}})+c_1(F_{\pm}|_{\kd_{\pm}})^2,$$
we need:
$$c_1(F_{\pm}|_{\kd_{\pm}})^2\leq -8\ \text{and}\ c_1(F_{\pm}|_{\kd_{\pm}})^2\equiv 0\mod 4.$$
Twisting $F_{\pm}$ by a line bundle, we can always assume that $c_1(F_{\pm}|_{\kd_{\pm}})$ is primitive in $N_0$.
Therefore, if the lattice $N_0$ has rank $1$, it must be generated by an element of square at most $-8$ and divisible by $4$.
\end{rmk}
\begin{rmk}\label{ChoiceY}
From condition (v) of Summary \ref{generaltec}, we see that it is convenient to
have an element in the lattice $N_+$ of square $-4$. Together
with the conditions of Remark \ref{RemarkImportante}, this is why we consider  $\CP^1\times\CP^2$: its Picard lattice contains elements of square $-4$, and it matches its double cover branched over
a $(2,2)$ divisor with $N_0\simeq(-72)$.
Looking at Table 2 of \cite{Crowley}, another possibility would be the pair of matching Fano $3$-folds numbered 25 and 14, given by
the blow-up of $\CP^3$ on an elliptic curve that is the intersection of two
quadrics and the blow-up of $V_5$ (section of the Pl\"ucker-embedded Grassmannian
$Gr(2,5)\subset\CP^9$ by a subspace of codimension $3$) on an elliptic curve that
is the intersection of two hyperplane sections. However, we did not manage to find examples with these blocks because of the restrictive condition (v). On the other hand, we found 6 other possible matchings with the suitable conditions, considering semi-Fano 3-folds of rank 2.
\end{rmk}
In conclusion, while  our approach does produce an original solution to the transversal gluing problem, it does not lend itself to the immediate mass-production  of examples. In order to achieve that, the main constraint  (v) should be suitably relaxed, possibly by  a finer Hartshorne--Serre  theorem or by an improved technique to get transversality.
The reader who might wish to join in the effort can follow this recipe: 
\begin{description}
\item[Step 1.]
Find two matching $N_\pm$-polarized semi-Fano $3$-folds $Y_\pm$ such that:
\begin{itemize}
\item[(i)] 
there exists $x\in N_+$ such $x^2=-4$ 
\item[(ii)]
there exists a primitive element $y\in N_0$ such that $y^2\leq -8$ and $4$ divides $y^2$.
\end{itemize}
\item[Step 2.]
Find $\mathcal{L}_\pm$ and $W_-$ which verify the conditions of Summary \ref{generaltec} (perhaps with  a computer).
\item[Step 3.]
The following must be checked by  \emph{ad-hoc} methods:
        \begin{enumerate}
        \item
         $H^2(\mathcal{L}^*_\pm)=0$, for the Hartshorne-Serre construction [Theorem \ref{thm: Hartshorne-Serre}];
         \item
         $H^1(\mathcal{L}^*_+)=0$ for our transversality method;
        \item
        that divisors with small slope do not contain $W$, for asymptotic stability [Proposition \ref{Hartshorne} (ii)];
        \item
        $H^1(E)=0$ for inelasticity [Corollary \ref{CInelastic2}].        
        \end{enumerate}
\item[Step 4.]
Conclude with similar arguments as in the proofs of Theorem \ref{thm:summaryintro} and Theorem \ref{thm:transv}.
\end{description}

\subsection{Construction of \texorpdfstring{$E_+$}{E+} over  
\texorpdfstring{$Y_+=\CP^1\times\CP^2$}
{Y+ = P1 x P2}
} 
\label{construction}
In this section all the objects considered are related to the building block $(Z_+,S_+)$  obtained by blowing up $Y_+=\CP^1\times\CP^2$ from  Example \ref{ex:2conics}.
We omit for simplicity the $+$ subscript.

In view of the constraints in Summary \ref{generaltec}, we apply Theorem \ref{thm: Hartshorne-Serre} to $Z=\Bl_\ccal Y$ as above, with parameters $$
W=\ell
\qand\mathcal{L}=\mathcal{O}_{Z}(-S-G+H).
$$
\begin{prop}\label{prop:exa}
Let $(Z,S)$ be a building block  
as in Example \ref{ex:2conics},
$\ccal$ a pencil base locus and $\ell\subset Z$ an exceptional fibre of  $\wt{\ccal}\rightarrow\mathscr{C}$.
There exists a rank $2$ Hartshorne-Serre  bundle $E\to Z$ obtained from $\ell$ such that:
\begin{enumerate}
\item
$c_{1}(E)=-\kd-G+H$,
\item
$E$ has a global section with vanishing locus $\ell$.
\end{enumerate}
\end{prop}
We start the proof of Proposition \ref{prop:exa} with a basic  lemma that will be invoked several times later on.
\begin{lemme}\label{lemmeBasic}
Let $X$ be a complex manifold and $D$ be an effective prime
divisor. 
\begin{enumerate}
        \item
        If $X$ is simply-connected, then $H^1(X,\mathcal{O}_X(-D))=0$.
        \item
        If $D$ is simply-connected and $X$ has no global holomorphic 2-form, then $H^2(X,\mathcal{O}_X(-D))=0$.
        \item
        If $X$ is a K3 surface, then $D^2\geq-2$.
\end{enumerate}
\end{lemme}
\begin{proof}
Items (i) and (ii) follow immediately from the exact sequence
$$
\xymatrix{ 0\ar[r]&\mathcal{O}_{Z}(-D)\ar[r] & \mathcal{O}_{Z}\ar[r] & \mathcal{O}_{D}\ar[r]& 0.}
$$
Item (iii) is straightforward  from Riemann-Roch:
$$
\frac{1}{2}D^2+2
  =\chi(\mathcal{O}_X(D))
  = h^0-h^1+h^2,
$$
where $h^0\geq1$ because $D$ is effective, $h^1=0$ by (i) and $h^2=0$ by Serre duality.
\end{proof}

To conclude the proof of Proposition \ref{prop:exa}, we apply Theorem \ref{thm: Hartshorne-Serre} using the following:

\begin{lemme}
\label{lemma:Hart-Serre+}
In the hypotheses of Proposition \ref{prop:exa},
\begin{enumerate}
        \item
        \label{H1lem}
        $H^{i}(\mathcal{O}_{Z}(\kd+G-H))=0$, 
        for
        $i=1,2.$
        
        \item
        \label{N}
        $\mathcal{L}_{|\ell}=\wedge^{2}\mathcal{N}_{\ell/Z}=\mathcal{O}_{\ell}(-1)$.
\end{enumerate}

\begin{proof}\quad
\begin{enumerate}
        \item
In view of  the exact sequence
$$
\xymatrix{ 
0\ar[r]&\mathcal{O}_{Z}(G-H)\ar[r] & \mathcal{O}_{Z}(\kd+G-H)\ar[r] & \mathcal{O}_{\kd}(A-B)\ar[r]& 0,
}
$$
and Serre duality, it suffices to check that 
 $H^{i}(Z,\mathcal{O}_{Z}(G-H))=0$ and $H^{i}(\kd,\mathcal{O}_{\kd}(B-A))=0$ , for $i\in\left\{1,2\right\}$.
 The latter is trivial, because neither $B-A$ nor $A-B$ are effective divisors. Moreover by Riemann--Roch, we also have $H^{1}(\kd,\mathcal{O}_{\kd}(B-A))=0$.
As to the former, the divisor $G$ is the class of a blow-up of $\CP^2$ on 4 points. Denoting by $h$ the pull back in $G$ of the class of a line in $\CP^2$, as before, we have:
$$\xymatrix{ 0\ar[r]&\mathcal{O}_{Z}(-H)\ar[r] & \mathcal{O}_{Z}(G-H)\ar[r] & \mathcal{O}_{G}(-h)\ar[r]& 0,}$$
and, by Lemma \ref{lemmeBasic}, $H^{i}(Z,\mathcal{O}_{Z}(-H))=0$ and $H^{i}(G,\mathcal{O}_{G}(-h))=0$ for $i\in\left\{1,2\right\}$. 
        
        \item
        Clearly $c_{1}(\mathcal{O}_{Z}(-\kd-G+H)_{|\ell})=(-\kd-G+H)\cdot \ell=-1$.
Now, since $\ell$ is a line, we have $c_{1}(\mathcal{T}_{\ell})=2$; moreover, line bundles on $\ell$ are classified by their first Chern class, so it suffices to check that $c_{1}( \wedge^{2}\mathcal{N}_{\ell/Z})=-1$. Indeed, using $\kd\cdot \ell=1$, this follows by adjunction: 
\[
c_{1}(\wedge^{2}\mathcal{N}_{\ell/Z})
= 
c_{1}(\mathcal{N}_{\ell/Z})=c_{1}((\mathcal{T}_{Z})_{|\ell})-c_{1}(\mathcal{T}_{\ell})
=S\cdot \ell-2 =-1.
\qedhere        
\]

\end{enumerate}
\end{proof}
\end{lemme}

We now compute some topological facts about the Hartshorne-Serre
bundle $E$ we just constructed in Proposition \ref{prop:exa}. These will be essential for the inelasticity results in Section \ref{sec:inelasticity}
but not elsewhere, so one may wish to skim through the proof on a first
read.

Recall that, by Theorem \ref{thm: Hartshorne-Serre}, there is a global section
$s\in H^0(E)$ such that $s^{-1}(0)=\ell$, where  $\ell$ is a fibre of the map
$p_{1}:\wt{\ccal}\rightarrow \mathscr{C}$. Hence, we have the following exact
sequence:
\begin{equation}
\xymatrix{ 0\ar[r]&\mathcal{O}_{Z}\ar[r]^{s} & E\ar[r] & \mathcal{I}_{\ell}\otimes \mathcal{O}_{Z}(-S-G+H)\ar[r]& 0,}
\label{principale}
\end{equation}
where $\mathcal{I}_{\ell}$ is the ideal sheaf of $\ell$ in $Z$. 
\begin{lemme}\label{HE}
We have $H^{0}(E)=\C$ and $H^{1}(E)=0$.
\end{lemme}
\begin{proof}
That $H^{0}(E)=\C$ follows directly from (\ref{principale}), since $-\kd-G+H$ is not an effective divisor and so $H^0(\mathcal{O}_{Z}(-\kd-G+H))=0$.

Similarly, since building blocks are simply-connected, the vanishing of $H^1(E)$ reduces to  that of 
$H^1(\mathcal{I}_{\ell}\otimes\mathcal{O}_{Z}(-\kd-G+H))$.
Twisting by $\mathcal{O}_{Z}(-\kd-G+H)$ the structural exact sequence of $\ell$ in $Z$, we have
\[
\xymatrix{ 
0\ar[r]&\mathcal{I}_{\ell}\otimes\mathcal{O}_{Z}(-\kd-G+H)\ar[r] &\mathcal{O}_{Z}(-\kd-G+H) \ar[r] & \mathcal{O}_{\ell}(-1)\ar[r]& 0,
}
\]
so we only have to establish that $H^1(\mathcal{O}_{Z}(-\kd-G+H))=0$.
In the exact sequence
$$
\xymatrix{ 
0\ar[r]&\mathcal{O}_{Z}(-\kd-G+H)\ar[r] &\mathcal{O}_{Z}(-G+H) \ar[r] & \mathcal{O}_{\kd}(-A+B)\ar[r]&0
}
$$
the divisor $-A+B$ is not effective, so $H^0(\mathcal{O}_{\kd}(-A+B))=0$. On the other hand, the divisor $H$ is the class of a blow-up of $\CP^1\times\CP^1$ on 12 points. Denoting by $h$ and $g$ the classes of the pull-back to $H$ of the lines in $\CP^1\times\CP^1$, respectively, the group $H^1(\mathcal{O}_{Z}(-G+H))$ must vanish by Lemma \ref{lemmeBasic} and the following exact sequences:
\[
\xymatrix{ 
0\ar[r]&\mathcal{O}_{Z}(-G)\ar[r] &\mathcal{O}_{Z}(-G+H) \ar[r] & \mathcal{O}_{H}(g-h)\ar[r]& 0;
}
\]
\[
\xymatrix{ 
0\ar[r]&\mathcal{O}_{H}(-h)\ar[r] &\mathcal{O}_{H}(g-h) \ar[r] & \mathcal{O}_{\CP^1}(-1)\ar[r]& 0.
}\qedhere
\]
\end{proof}

\subsection{Construction of \texorpdfstring{$E_-$}{E-} over 
\texorpdfstring{$Y_-\protect\overset{2:1}{\longrightarrow}\CP^1\times\CP^2$}
{Y-  -> P1 x P2}
} 
\label{construction2}

Similarly, in this section all the objects considered are related to building block $Z_-$ obtained by blowing up  $Y_-\overset{2:1}{\longrightarrow}\mathbb{P}^1\times\mathbb{P}^2$ from Example  \ref{ex:2conics2}. We also omit the $-$ subscript.

We apply Theorem \ref{thm: Hartshorne-Serre} to $Z$ as above, with
\[
[W]=h
\qand
\mathcal{L}=\mathcal{O}_{Z}(G).
\]
See Example \ref{ex:2conics2} for the notation. (As described there, the
possible choices of the line $W$ are parametrised by an open subset of a surface
$\widetilde Q \subset \kd$.)

\begin{prop}\label{prop:exa2}
Let $(Z,S)$ be a building block 
provided in Example \ref{ex:2conics2} and $W$ a line of class $h$.
There exists a rank 2 Hartshorne-Serre  bundle $E\to Z$ obtained from $W$ such that:
\begin{enumerate}
\item
$c_{1}(E)=G$,
\item
$E$ has a global section with vanishing locus $W$.
\end{enumerate}
\end{prop}

As before, Proposition \ref{prop:exa2} is a direct application of Theorem \ref{thm: Hartshorne-Serre}, using: 

\begin{lemme}
\label{lemma:Hart-Serre-}
In the hypotheses of Proposition \ref{prop:exa2},
\begin{enumerate}
        \item
        \label{H1lem2}
        $H^{1}(\mathcal{O}_{Z}(-G))=H^{2}(\mathcal{O}_{Z}(-G))=0.$
        
        \item
        \label{N2}
        $\mathcal{L}_{|W}=\wedge^{2}\mathcal{N}_{W/Z}=\mathcal{O}_{W}$.

\end{enumerate}
\begin{proof}\quad
\begin{enumerate}
        \item
        This is immediate from Lemma \ref{lemmeBasic}.
        
        \item
        We  proceed as in the proof of Lemma \ref{lemma:Hart-Serre+} \ref{N}. 
On one hand we have $c_{1}(\mathcal{O}_{Z}(G)_{|W})=0$, because $G\cdot W=0$.
On the other hand, since $c_{1}(\mathcal{T}_{Z})= 2H+G-\wt{\ccal}$ we have $$c_{1}((\mathcal{T}_{Z})_{|W})=(2H+G-\wt{\ccal})\cdot W=2.$$
Since $W$ is a line, $c_{1}(T_{W})=2$. It follows that $c_{1}(\wedge^{2}\mathcal{N}_{W/Z})=c_{1}((\mathcal{T}_{Z})_{|W})-c_{1}(\mathcal{T}_{W})=0$.
\qedhere
\end{enumerate}
\end{proof}
\end{lemme}

Again, the following topological facts about the Hartshorne-Serre
bundle $E_-$ from Proposition \ref{prop:exa2} will be used in Section \ref{sec:inelasticity}.

By Theorem \ref{thm: Hartshorne-Serre}, there is a global section $s\in H^0(E)$ such that $(s)_{0}=W$ is a line of class $h$. Hence, we have the following exact sequence:
\begin{equation}
\label{principale2}
        \xymatrix{ 
        0\ar[r]&\mathcal{O}_{Z}\ar[r]^{s} & E\ar[r] & \mathcal{I}_{W}\otimes         \mathcal{O}_{Z}(G)\ar[r]& 0,
        }
\end{equation}
where $\mathcal{I}_{W}$ is the ideal sheaf of $W$ in $Z$. 
\begin{lemme}\label{HE2}
We have $H^{0}(E)=\C^2$ and $H^{1}(E)=0$.
\begin{proof}
We follow the same approach as in the proof of Lemma \ref{HE}. That $H^{0}(E)=\C^2$ reduces, by (\ref{principale2}), to the fact that $H^0(\mathcal{I}_{W}\otimes\mathcal{O}_{Z}(G))=\C$, since $H^0(\mathcal{O}_{Z})=\C$ and $H^1(\mathcal{O}_{Z})=0$. Indeed, there is only one  global section of $\mathcal{O}_{Z}(G)$ that vanishes on the line $W$.

Similarly,  for the vanishing of $H^1(E)$, it suffices to check that 
$H^1(\mathcal{I}_{W}\otimes\mathcal{O}_{Z}(G))=0$. Twisting by $\mathcal{O}_{Z}(G)$ the structural  exact sequence of $W$ in $Z$, we have
$$\xymatrix{ 0\ar[r]&\mathcal{I}_{W}\otimes\mathcal{O}_{Z}(G)\ar[r] &\mathcal{O}_{Z}(G) \ar[r] & \mathcal{O}_{W}\ar[r]& 0.}$$
Since $H^0(\mathcal{I}_{W}\otimes\mathcal{O}_{Z}(G))=\C$, $H^0(\mathcal{O}_{Z}(G))=\C^2$ and $H^0(\mathcal{O}_{W})=\C$,
the map $H^0(\mathcal{O}_{Z}(G))\rightarrow H^0(\mathcal{O}_{W})$ is necessarily surjective.
So, we only have to prove that $H^1(\mathcal{O}_{Z}(G))=0$, which is clear from the  exact sequence
\[
\xymatrix{ 
        0\ar[r]&\mathcal{O}_{Z}\ar[r] &\mathcal{O}_{Z}(G) \ar[r] & \mathcal{O}_{G}\ar[r]&0.
}\qedhere
\]
\end{proof}
\end{lemme}

\section{The moduli space of stable bundles on \texorpdfstring{$\kd$}{S}}

In Section \ref{subsec: stabi1S}, we deduce the asymptotic
stability of $E_\pm$ (Propositions \ref{stabi1} and \ref{stabi2}), as well as the dimension of the corresponding moduli space at infinity (Proposition \ref{qui}). In  Section \ref{sec moduli assoc to E|S}, we  establish the freedom to choose the base locus curve $\ccal_+$ in order to match any given asymptotic incidence condition  (Theorem \ref{thm:summaryintro}).

We begin by recalling some known facts on moduli spaces of semi-stable sheaves on a K3 surface $S$ (see \cite{Huybrechts}).  
A \emph{Mukai vector} is a triple  
$$
v=(r,l,s)\in  \left(H^{0}\oplus H^{2}\oplus H^{4}\right)(\kd,\Z).
$$ 
We define a pairing between Mukai vectors  $(r,l,s)$
and $(r',l',s')$ as follows:
$$(r,l,s)\cdot(r',l',s'):=l\cdot l'-rs'-r's.$$
The Mukai vector of  a vector
bundle $E\to \kd$ is
defined as
\[
v(E):=\left(\rk E,\, c_{1}(E),\,\chi(E)-\rk E \right),
\] 
with
$\chi(E)=\frac{c_{1}(E)^2}{2}+2\rk E-c_{2}(E)$. 

The local structure of the moduli space of stable bundles over a K3 surface $S$ can be computed in several ways, which trace back to the work of Maruyama (see \cite[Proposition 6.9]{Maruyama2}):

\begin{thm}[Maruyama]
\label{thmMaruyama}
Let  $L\to S$ be a polarised K3 surface and denote by $\mcal^{s}_{\kd,L}(v)$ the moduli space of isomorphism classes of $L$-slope-stable vector bundles on $S$ with Mukai vector $v$.
  If  $\mcal^{s}_{\kd,L}(v)$ is not empty, then it is a quasi-projective complex manifold of dimension $v^{2}+2$ and its Zariski tangent space at a point $E$ admits the following isomorphisms:
$$T_{E}\mcal^{s}_{\kd,L}(v)=\Ext^1(E,E)=H^{1}(\Endr(E)).$$
Furthermore,
\begin{align*}
\dim \mcal^{s}_{\kd,L}(v)=-\chi(\Endr_{0}(E))&=2(\rk E)^2-2\chi(E)\rk E+c_1(E)^2\\
&=(1-\rk(E))c_{1}(E)^{2}+2(\rk E)c_{2}(E)-2(\rk E)^{2}+2.
\end{align*}
\end{thm}

\noindent NB.: We recall that a polarisation over a nonsingular projective variety $Y$ is determined by an ample line bundle $L\to Y$. 
The  \emph{$L-$degree} and the \emph{$L$--slope} of a coherent sheaf $E\to Y$ are, respectively,
\begin{equation}        \label{eq: mu_L}
    \degl E:= c_1(E)\cdot L^{\dim Y-1}
\qand
    \mu_L( E) := 
    \frac{\degl E}{\rk(E)} .
\end{equation}
Then $E $ is \emph{(semi-)slope-stable} if, for every proper coherent subsheaf $F \subset E$ such that $E/F$ is torsion-free, one has
$$
\mu_L (F) \underset{(\leq)}{<} \mu_L (E).
$$
If $E$ is locally free,   it suffices to check stability for all \emph{reflexive} subsheaves $F\subset E$.

\subsection{Asymptotic stability of the Hartshorne-Serre bundles \texorpdfstring{$E_\pm$}{E+-}}
\label{subsec: stabi1S}

We need suitable stability
criteria for bundles over  $S_\pm$. Following \cite{Henrique}, a variety $Y$ is called \emph{polycyclic} if its Picard group is free Abelian. Given a polarisation $L\to Y$,  the\emph{ $L$-degree} of a divisor $D\subset \Pic(Y)$ is [cf.  (\ref{eq: mu_L})]
$$
\delta_L (D):=\degl\mathcal{O}_Y(D). 
$$

\begin{cor}[{\cite[Corollary 4]{Henrique}}]
\label{Hoppbis}
Let  $\mathcal{G} \rightarrow Y$ be a holomorphic vector bundle of rank 2 over a polycyclic variety with $\Pic(Y)\simeq \Z^{l+1}$ and polarisation $L$. 

The bundle $\mathcal{G}$ is (semi)-stable if and only if 
$$
H^{0}( \mathcal{G}\otimes\mathcal{O}_Y(D)) = 0
$$
for all $D\in \Pic(Y)$ such that 
\begin{align*}
\delta_L(D)  &\leq -\mu_{L}(\mathcal{G}). \\
&(<)
\end{align*}
\end{cor}

\begin{prop}[{\cite[Proposition 10]{Henrique}}]
\label{Hartshorne}
Let $Y$ be a smooth polycyclic variety endowed with a polarization $L$.
Let $E\to Y$ be a rank $2$ Hartshorne--Serre bundle obtained from some   $W\subset Y$ as in Theorem \ref{thm: Hartshorne-Serre}.
Then $E$ is stable (resp. semi-stable) if 
\begin{enumerate}
\item
$\mu_{L}(E)>0$ (resp. $\mu_{L}(E)\geq0$), and
\item
for all hyper-surfaces $\mathcal{\kd}$ with $\delta_L(\mathcal{\kd})  \leq \mu_{L}(E)$ (resp. 
$\delta_L(\mathcal{\kd})<\mu_{L}(E)$) the subscheme $W$ is not contained in $\mathcal{\kd}$.
\end{enumerate}

\end{prop}

We may now apply the above general criterion to both sides of our present setup.

\begin{prop}\label{stabi1}
Let $(Z_+,\kd_+)$ be a building block and $\ccal_+$ a pencil base locus provided in Example \ref{ex:2conics}. Let $E_+\to Z_+$ be given by
Proposition \ref{prop:exa}, such that

\begin{enumerate}
\item
$c_{1}(E_+)=-\kd_+-G_++H_+$, and
\item
$E_+$ has a global section whose vanishing locus is a fibre $\ell_+$  of $p_1:\wt{\ccal}\rightarrow \mathscr{C}$.
\end{enumerate}
Then  $E_{+|\kd_+}$ is stable.
\end{prop}
\begin{proof}
The bundle $E_{+|\kd_+}$ can also be seen as a Hartshorne--Serre construction. Indeed, restricting
the exact sequence (\ref{principale}), we obtain:
\begin{equation}
\xymatrix{ 0\ar[r]&\mathcal{O}_{\kd_+}\ar[r] & E_{+|\kd_+}\ar[r] & \mathcal{I}_{p}\otimes\mathcal{O}_{\kd_+}(B_+-A_+)\ar[r]& 0,}
\label{principaleS}
\end{equation}
where $p:=p_{1}(\ell)$ is the projection of $\ell$ on $\ccal$. To prove stability using Proposition \ref{Hartshorne}, we only have to check that $\kd_+$ does not contain any effective divisor $D$ of degree
$$
\delta_{\mathcal{A}_+}(D)
\leq
\mu_{\mathcal{A}_+}(E_{+|\kd_+})=\frac{(A_++B_+)\cdot (B_+-A_+)}{2}=1.
$$
Suppose such a divisor $D=\alpha A_++\beta B_+$ exists; since  $D$ is effective, we actually have
$$
1=\delta_{\mathcal{A}_+}(D)=(\alpha A_++\beta B_+)\cdot(A_++B_+)=5\beta+3\alpha.
$$
Moreover, $D$ is necessarily a prime divisor, for the sum of two effective divisors would have degree at least $2$. 
By Lemma \ref{lemmeBasic} (iii), we also have 
$$
2\beta^2+6\alpha\beta=D^2\geq-2.
$$
Hence
$\frac{1-\sqrt{17}}{8}\leq\beta\leq \frac{1+\sqrt{17}}{8}$,
and the only integer solution  $\beta=0$  implies $\alpha=\frac{1}{3}\notin\Z$.
\end{proof}
\begin{prop}\label{stabi2}
Let $(Z_-,\kd_-)$ be a building block provided in Example \ref{ex:2conics2}. Let $E_-$ be a bundle on $Z_-$ constructed in
Proposition \ref{prop:exa2} such that

\begin{enumerate}
\item
$c_{1}(E_-)=G_-$, and
\item
$E_-$ has a global section whose vanishing locus is $W$, where $\left[W\right]=h_-$.
\end{enumerate}
The bundle $E_{-|\kd_-}$ is stable.
\end{prop}
\begin{proof}
We proceed as in the proof of Proposition \ref{stabi1}.
The bundle $E_{-|\kd_-}$ can also be seen as a Hartshorne--Serre construction by restricting
 (\ref{principale2}). Thus we
must check that $\kd_-$ does not contain any effective divisor $D$ of degree
$$
\delta_{\mathcal{A}_-}(D)
\leq
\mu_{\mathcal{A}_-}(E_{-|\kd_-})=\frac{(2A_-+B_-)\cdot A_-}{2}=2.
$$
Suppose such $D=\alpha A_-+\beta B_-$ exists; since the intersection form on $\Pic \kd_-$ is even and $D$ is effective, we have $\delta_{\mathcal{A}_-}(D)= 2$ and so
$$
1=\frac{1}{2}\delta_{\mathcal{A}_-}(D)
=\frac{1}{2}(\alpha A_-+\beta B_-)\cdot(2A_-+B_-)
= 5\beta+2\alpha.
$$
Moreover, $D$ is also  prime, for otherwise its degree would be at least 4, and so
$$2\beta^2+8\alpha\beta=D^2\geq-2.$$
Hence $\frac{1-\sqrt{40}}{18}\leq\beta\leq \frac{1+\sqrt{40}}{18}
\Rightarrow \beta=0
\Rightarrow \alpha=\frac{1}{2} \notin \Z$. 
\end{proof}

In the context above, the moduli spaces of the stable bundles $E_{\pm|\kd_\pm}$
have `minimal' positive dimension:
\begin{prop}\label{qui}
Let $(Z_\pm,\kd_\pm)$ be the building blocks provided in Examples \ref{ex:2conics} and \ref{ex:2conics2}, and let  $E_\pm\to Z_\pm$ be the asymptotically stable bundles constructed in
Propositions \ref{prop:exa} and \ref{prop:exa2}.
Let $\mcal^{s}_{\kd_\pm,\mathcal{A}_\pm}(v_\pm)$ be the moduli space of $\mathcal{A}_\pm$-stable bundles on $\kd_\pm$ with Mukai vector $v_\pm=v(E_{\pm|\kd_\pm})$. We have:
$$\dim \mcal^{s}_{\kd_\pm,\mathcal{A}_\pm}(v_\pm)=2.$$
\begin{proof}
That $E_\pm$ are asymptotically stable is the content of the previous Propositions \ref{stabi1} and \ref{stabi2}. Now the claim is a direct application of  Theorem \ref{thmMaruyama}, with $\rk E_{\pm|\kd_\pm}=2$, $c_{1}(E_{+|\kd_+})^{2}=-4$, $c_{2}(E_{+|\kd_+})=1$, $c_{1}(E_{-|\kd_-})^{2}=0$, and $c_{2}(E_{-|\kd_-})=2$.
\end{proof}
\end{prop}

\subsection{Proof of Theorem  \ref{thm:summaryintro}}
\label{sec moduli assoc to E|S}

Let $Y_+=\CP^1\times\CP^2$ as in
Example \ref{ex:2conics}, and
$\kd_+ \subset Y_+$ be a smooth anti-canonical K3 divisor.
Let $\mathcal{A}_+$ be the ample class $A_++B_+$ on $\kd_+$,
$v_{\kd_+}$ the Mukai vector $(2, \, B_+-A_+, \, -1)$.
The associated moduli space $\mcal^{s}_{\kd_+,\mathcal{A}_+}(v_{\kd_+})$ is
2-dimensional.

For a smooth curve $\ccal_+ \in |{-}K_{Y_+|\kd_+}|$, let $Z_+ := \textrm{Bl}_{\ccal_+} Y_+$
be the building block resulting from Proposition \ref{FanoBlock}, and
let 
$$
v_{Z_+}:=(2, -\kd_+-G_++H_+, \ell_+)\in \left(H^{0}\oplus H^{2}\oplus H^{4}\right)(Z_+,\Z).
$$
Given a bundle $E_+\to Z_+$ as in Proposition \ref{stabi1} with $(\rk E_+, c_{1}(E_+),c_{2}(E_+))=v_{Z_+}$, the restriction
to $\kd_+$ has Mukai vector $v_{\kd_+}$, so $\mathcal{G} := E_{+|\kd_+} \in \mcal^{s}_{\kd_+,\mathcal{A}_+}(v_{\kd_+})$.

We have now established all the preliminaries for Theorem \ref{thm:summaryintro}, and the rest of this section is devoted to its proof. Since all relevant objects are associated to the block $(Z_+,S_+)$, we omit henceforth the  $+$ subscript.
  
Given $\ccal\in|{-}K_{Y|\kd}|$,
we have used the Hartshorne-Serre construction to construct a
family of vector bundles $\{E_p\to Z \mid p \in \ccal\}$ with 
\[
(\rk E, c_{1}(E),c_{2}(E))=v_Z
\]
parametrised by $\ccal$ itself. Proposition \ref{stabi1} showed that each
$E_p$ is asymptotically stable.

Moreover, Proposition \ref{inelasticity} in the next section will show that $E_p$ is inelastic.

\begin{lemme}\label{summary2}
For each $p\in \kd$, there exists a rank $2$ Hartshorne-Serre  bundle $\mathcal{G}_p\to \kd$ obtained from $p$ such that:
\begin{enumerate}
\item
$c_{1}(\mathcal{G}_p)=B-A$,
\item
$\mathcal{G}_p$ has a unique global section (up to scale) with vanishing locus $p$.
\item
$\mathcal{G}_p$ is $\mathcal{A}$-$\mu$-stable.
\end{enumerate}

\begin{proof}
By Serre duality, $H^2(S,A-B)=H^0(S,B-A)$, which vanishes since $B-A$ is not an effective divisor. Then a bundle $\mathcal{G}_p$ satisfying (i) is given by Theorem \ref{thm: Hartshorne-Serre} and it fits in the exact sequence
\begin{equation}
\xymatrix{ 0\ar[r]&\mathcal{O}_{\kd}\ar[r] & \mathcal{G}_p\ar[r] & \mathcal{I}_{p}\otimes\mathcal{O}_{\kd}(B-A)\ar[r]& 0.}
\label{exactG}
\end{equation}
Again since $B-A$ is not effective, the sheaf $\mathcal{I}_{p}\otimes\mathcal{O}_{\kd}(B-A)$ has no global sections and (ii) follows trivially from (\ref{exactG}).

The stability of $\mathcal{G}_p$ is equivalent to the stability of $E_{+|\kd_+}$ proven in Proposition \ref{stabi1}, since they are both extensions of $\mathcal{O}_{\kd}$ and $\mathcal{I}_{p}\otimes\mathcal{O}_{\kd}(B-A)$.
\end{proof}
\end{lemme}

One crucial feature of the building block obtained from $Y_+=\PP^1\times\PP^2$ is the fact that the moduli space of bundles over the anti-canonical K3 divisor $S$ is actually isomorphic to $S$ itself: 
\begin{prop}\label{iso}
The map 
$$
\begin{array}{rcl}
g:S&\longrightarrow&\mcal_{\kd,\mathcal{A}}^{s}(v_\kd)\\
p&\longmapsto&\mathcal{G}_p
\end{array}
$$ defined by Lemma \ref{summary2}  is an isomorphism of K3 surfaces.
\begin{proof}
It is clear from Lemma \ref{summary2} (ii) that $g$ is injective, so the issue lies in the structure of the image.
 Our  $\mcal_{\kd,\mathcal{A}}^{s}(v_\kd)$ is an open subset of the moduli space  $\mcal_{\kd,\mathcal{A}}^{ss,G}(v_\kd)$ of Gieseker semi-stable sheaves on $\kd$, and the latter have first Chern class $B-A$ primitive in $\Pic \kd$. 
Hence, by  \cite[Theorem 6.2.5]{Huybrechts}, $\mcal_{\kd,\mathcal{A}}^{ss,G}(v_\kd)$ is a K3 surface if the polarisation $\mathcal{A}$ is contained in an open chamber (\cf \cite[Definition 4.C.1]{Huybrechts}), \ie, if 
$$
\mathcal{A}\cdot D\neq 0, \quad\forall D\in\Div(S) 
\text{ such that }-\Delta\leq D^2<0,
$$
where $\Delta:=2\rk .c_2-(\rk-1)c_1^2$ is the discriminant in $\mcal_{\kd,\mathcal{A}}^{ss,G}(v_\kd)$.
In our case $\Delta=4-(-4)=8$, and again we argue by contradiction as in Proposition \ref{stabi1}: suppose there is a divisor $D=\alpha A+\beta B$ such that $-8\leq D^2<0$ but $D\cdot \mathcal{A}= 0$; then 
\begin{eqnarray*}
\begin{cases}
-4\leq\beta^2+3\alpha\beta <0\\
5\beta+3\alpha=0
\end{cases}
&\Rightarrow&
0<\beta^2\leq1.
\end{eqnarray*}
The integer solutions $\beta=\pm1$ imply $\alpha=\mp\frac{5}{3}\notin\Z$, therefore $\mcal_{\kd,\mathcal{A}}^{ss,G}(v_\kd)$ is a K3 surface. It follows that the map $g$ is a bi-meromorphism of K3 surfaces between $\kd$ and $\mcal_{\kd,\mathcal{A}}^{ss,G}(v_\kd)$, and every such map is an isomorphism. It follows that $\mcal_{\kd,\mathcal{A}}^{s}(v_\kd)=\mcal_{\kd,\mathcal{A}}^{ss,G}(v_\kd)$.
\end{proof}
\end{prop}

Now let
$\mathcal{G} \in  \mcal_{\kd,\mathcal{A}}^{s}(v_\kd)$ and
 $V \subset H^1(\kd, \Endr_0(\mathcal{G}))$.
 From Proposition \ref{iso}, there is $p\in\kd$ such that $\mathcal{G}=\mathcal{G}_p$ and let $V'=(\mathrm{d} g)_p ^{-1}(V)$.
 Since ${-}K_{Y|\kd}$ is very ample (see Example \ref{ex:2conics}), Lemma \ref{lem:wiggle} allows the choice of
a smooth base locus curve $\ccal \in |{-}K_{Y|\kd}|$ such that $p\in \ccal$ and $T_p \ccal = V'$.
By Proposition \ref{prop:exa}, we can find a family  $\{E_q\to Z \mid q \in \ccal\}$ of bundles parametrised by $\ccal$, with prescribed topology
$$
(\rk E, c_{1}(E),c_{2}(E))=v_Z
$$
and $E_{q|\kd}=\mathcal{G}_q$.
The bundle $E_p$ has therefore all the properties claimed in Theorem  
\ref{thm:summaryintro} apart the inelasticity which will be proved in the next section.

\section{Inelasticity of asymptotically stable Hartshorne-Serre bundles}\label{sec:inelasticity}
\begin{defi}
Let $(Z,\kd)$ be a building block and $E$ a bundle on $Z$.
We say that $E$ is \emph{inelastic} if 
$$H^{1}(Z,\Endr_{0}(E)(-\kd))=0.$$
\end{defi}
This condition means that there are no global deformations of the bundle $E$
which keep fixed the bundle ``at infinity'' $E_{|S}$.
Section \ref{sec: inelasticity of AS bundles} provides a characterisation of
inelasticity in the case of asymptotically stable bundles, for then one may
relate the freedom to extend $E$ and the dimension of the moduli space
$\mcal^{s}_{\kd,\mathcal{A}}(v_E)$.
In Section \ref{sec: Inelasticity of H-S bundles} we apply this to
Hartshorne-Serre bundles, by computing the dimension of the moduli space in
terms of the construction data. These results hold for general building blocks and may be of independent interest. 

Section \ref{sec: inelast E+} 
contains the
computations in cohomology to establish the inelasticity of our bundles $E_\pm$ constructed in Propositions \ref{prop:exa} and \ref{prop:exa2}. 

\subsection{Inelasticity of asymptotically stable bundles}
\label{sec: inelasticity of AS bundles}

This section is dedicated to proving the following statement.
\begin{prop}\label{Inelasticity1}
Let $(Z,\kd)$ be a building block and $E$ an asymptotically stable bundle on $Z$. Let $\mcal^{s}_{\kd,\mathcal{A}}(v)$ be the moduli space of $\mathcal{A}$-$\mu$-stable bundles on $\kd$ with Mukai vector $v=v(E_{|\kd})$.
The following statements are equivalent:
\begin{enumerate}
 \item
 The bundle $E$ is inelastic.
 \item
 The sequence 
 $$
\xymatrix@R0pt{ 0\ar[r]&\Ext^{1}(E,E)\ar[r] &\Ext^{1}(E_{|\kd},E_{|\kd}) 
 \ar[r]& H^2(Z,\Endr(E)(-\kd))\ar[r]& 0}.
$$
(which is self-dual for Serre duality) is exact. 
 \item
 $\dim \Ext^{1}(E,E)=\frac{1}{2}\dim \mcal^{s}_{\kd,\mathcal{A}}(v).$
\end{enumerate}
\end{prop}

By Serre duality we have $\chi(\Endr_{0}(E)(-\kd))=-\chi(\Endr_{0}(E))$.
Now, restriction to $S$  gives the exact sequence 
\begin{equation}
        \xymatrix@R0pt{ 0\ar[r]&\Endr_{0}(E)(-\kd)\ar[r] & \Endr_{0}(E)
        \ar[r]& \Endr_{0}(E)_{|\kd}\ar[r]& 0,}
        \label{sequ1I}
\end{equation}
hence, by Maruyama's Theorem \ref{thmMaruyama}, it is also equivalent to: 
\begin{equation}\label{eq: dim = 2chi}
2\chi(\Endr_{0}(E)(-\kd))
=-\chi(\Endr_{0}(E)_{|\kd})
=\dim \mcal^{s}_{\kd,\mathcal{A}}(v).
\end{equation}
Moreover, the long exact sequence associated to (\ref{sequ1I}) and the Serre duality show that 
$E$ being inelastic is equivalent of having the following exact sequence:
\begin{equation}
        \xymatrix@R0pt{ 0\ar[r]&H^1(Z, \Endr_{0}(E))\ar[r] & H^1(\kd, \Endr_{0}(E)_{|\kd})
        \ar[r]& H^2(Z, \Endr_{0}(E)(-\kd))\ar[r]& 0.}
        \label{sequ1II}
\end{equation}

\begin{lemme}\label{Esimple}
The bundle $E$ is simple.
\begin{proof}
The restriction of the class of $\kd$ to $\kd$ is trivial. Hence twisting  (\ref{sequ1I}) by $\mathcal{O}_{Z}(-(n-1)\kd)$, with $n\in\mathbb{N}$, we get:
$$
\xymatrix@R0pt{ 0\ar[r]&\Endr_{0}(E)(-n\kd)\ar[r] & \Endr_{0}(E)(-(n-1)\kd)
 \ar[r]& \Endr_{0}(E)_{|\kd}\ar[r]& 0}.
$$
Since $E_{|\kd}$ is stable, in particular it is simple
and so $H^{0}(\Endr_{0}(E)_{|S})=0$.
It follows by induction that
$$
H^{0}(\Endr_{0}(E))=H^{0}(\Endr_{0}(E)(-n\kd)),
\quad\forall n\in \mathbb{N}.
$$
If there could occur $H^{0}(\Endr_{0}(E))\neq 0$, then one would have 
$$
h^{0}(\Endr_{0}(E))\geq h^{0}(\mathcal{O}_{Z}(n\kd)),
\quad\forall n\in \mathbb{N}.
$$
Considering the exact sequence
$$\xymatrix@R0pt{ 0\ar[r]&\mathcal{O}_{Z}((n-1)\kd)\ar[r] & \mathcal{O}_{Z}(n\kd)
 \ar[r]& \mathcal{O}_{\kd}\ar[r]& 0,}$$
we find, by induction, $h^{0}(\mathcal{O}_{Z}(n\kd))=n+1$,
which would render the dimension $h^{0}(\Endr_{0}(E))$ undefined; so indeed it must vanish.
\end{proof}
\end{lemme} 

We now examine the terms on the left-hand side of  (\ref{eq: dim = 2chi}). It follows from  (\ref{sequ1I}), Lemma \ref{Esimple} and Serre duality that $h^{0}$ and $h^{3}$ are zero:
$$
h^{0}(\Endr_{0}(E)(-\kd))=\underset{0}{\underbrace{h^{0}(\Endr_{0}(E))}}=h^{3}(\Endr_{0}(E)(-\kd)),
$$
therefore $h^1=h^2-\chi(\Endr_{0}(E)(-\kd)$. On the other hand, from the exact sequence
\begin{equation}
\xymatrix@R0pt{ 0\ar[r]&\Endr_{0}(E)\ar[r] & \Endr(E)
 \ar[r]^{\tr}& \mathcal{O}_{Z}\ar[r]& 0}
\label{sequ1III}
 \end{equation}
 it follows that $H^{1}(\Endr_{0}(E))=H^{1}(\Endr(E))=\Ext^{1}(E,E)$, and we conclude by Serre duality in $h^2$:  
$$
h^{1}(\Endr_{0}(E)(-\kd))=\dim\Ext^{1}(E,E)-\frac{1}{2}\dim \mcal^{s}_{\kd,\mathcal{A}}(v).
$$
This gives (i) $\Leftrightarrow$ (iii).

Moreover, applying Lemma \ref{lemmeBasic}, exact sequence (\ref{sequ1III}) tensorized by $\mathcal{O}_{Z}(-\kd)$ also provides
$H^{2}(\Endr_{0}(E)(-\kd))=H^{2}(\Endr(E)(-\kd))$. 
Similarly, from the exact sequence
$$\xymatrix@R0pt{ 0\ar[r]&\Endr_{0}(E_{|\kd})\ar[r] & \Endr(E_{|\kd})
 \ar[r]^{\tr}& \mathcal{O}_{\kd}\ar[r]& 0}$$
 we have $H^{1}(\Endr_{0}(E_{|\kd}))=H^{1}(\Endr(E_{|\kd}))=\Ext^{1}(E_{|\kd},E_{|\kd})$.
Then (\ref{sequ1II}) gives (i) $\Leftrightarrow$ (ii).

\subsection{Application to Hartshorne-Serre bundles}
\label{sec: Inelasticity of H-S bundles}

Here we  establish a characterisation of inelasticity in the case of
asymptotically stable Hartshorne-Serre bundles $E$ of rank $2$, as in the
context of Theorem \ref{thm:main}.
In view of Proposition \ref{Inelasticity1}, we can accomplish this by
calculating $\dim \Ext^{1}(E,E)$, the dimension of the infinitesimal
deformations of $E$. 

\begin{prop}\label{prop:dimext}
Let $(Z,\kd)$ be a building block, and let $E \to Z$ be an asymptotically
stable Hartshorne--Serre bundle obtained from a genus $0$ curve $W\subset Z$
and a line bundle $\mathcal{L}\to Z$ as in Theorem \ref{thm: Hartshorne-Serre}. 

Suppose $H^1(E)=0$. Then
\begin{equation}
\label{eq:dimext}
\dim \Ext^{1}(E,E) =
\dim H^0(W,\mathcal{N}_{W/Z}) + \dim H^1(Z,\mathcal{L}^*)-\dim H^0(Z,E)+1.
\end{equation}
\end{prop}
 
\begin{rmk}
The Hartshorne-Serre construction produces a vector bundle $E$ together
with a section $s$ (up to scale). The degrees of freedom in the
construction come from deformations of the curve $W$, parametrised by
$H^0(W,\mathcal{N}_{W/Z})$, and choosing an extension
\begin{equation}
\xymatrix{ 0\ar[r]&\mathcal{O}_{Z}\ar[r] & E\ar[r] & \mathcal{I}_{W}\otimes \mathcal{L}\ar[r]& 0,}
\label{PrincipalGene}
\end{equation}
parametrised by 
$\Ext^1(\mathcal{I}_W,\mathcal{L}^*)$.
Hence one would naively expect that
the space of pairs $(E,s)$ produced by the construction has dimension
$\dim H^0(W,\mathcal{N}_{W/Z}) + \dim \Ext^1(\mathcal{I}_W,\mathcal{L}^*)=\dim H^0(W,\mathcal{N}_{W/Z}) + \dim H^1(Z,\mathcal{L}^*)+1$.
Accounting for the choice of $s$ yields the right hand
side of \eqref{eq:dimext}, so the proposition amounts to stating that the
naive calculation gives the correct result under the given hypotheses.
\end{rmk}
 
The remainder of the section is devoted to the proof of
Proposition \ref{prop:dimext}.
By construction,  $E$ fits in the exact sequence \eqref{PrincipalGene}.
Applying the functor $\End(\cdot,E)$ we obtain:
\begin{equation}
\begin{split}
\xymatrix@R0pt{ 
        0\ar[r]&\End(\mathcal{I}_{W}\otimes \mathcal{L},E) \ar[r]& \End(E,E)\ar[r]&H^{0}(E)\\
        \ar[r]& \Ext^{1}(\mathcal{I}_{W}\otimes \mathcal{L},E)\ar[r] &\Ext^{1}(E,E)\ar[r] & H^{1}(E).
}
\end{split}
\label{sequ2I}
\end{equation}
Since $\Endr(\mathcal{I}_{W},\mathcal{O}_{Z})=\mathcal{O}_{Z}$,
it follows that $$\End(\mathcal{I}_{W}\otimes \mathcal{L},E)=H^0(E\otimes\mathcal{L}^*).$$
We first show that $H^0(E\otimes\mathcal{L}^*)=0$.
Twisting  (\ref{PrincipalGene}) by $\mathcal{L}^*$ we get
\begin{equation}
\label{eq:PrincipalGene x L*}
\xymatrix{ 
0\ar[r]&\mathcal{L}^*\ar[r] & E\otimes\mathcal{L}^*\ar[r] & \mathcal{I}_{W}\ar[r]&0.
}
\end{equation}
We know that $H^0(\mathcal{I}_{W})=0$. Since $E$ is asymptotically stable,  $\mathcal{L}$ corresponds necessarily to an effective divisor, so $H^0(\mathcal{L}^*)=0$ and the claim follows.

Moreover, by assumption, $H^1(E)=0$, so   (\ref{sequ2I}) simplifies to
$$
\xymatrix@C10pt{ 0\ar[r]& \End(E,E)\ar[r] & H^{0}(E)
\ar[r]& \Ext^{1}(\mathcal{I}_{W}\otimes\mathcal{L} ,E)\ar[r] &\Ext^{1}(E,E)\ar[r] & 0.}$$
Using 
Lemma \ref{Esimple}, this gives$$
\dim \Ext^{1}(E,E)
=\dim\Ext^{1}(\mathcal{I}_{W}\otimes\mathcal{L} ,E) -\dim H^0(E)+1,$$
and it only remains to check that
\begin{equation}
\dim\Ext^{1}(\mathcal{I}_{W}\otimes\mathcal{L} ,E)=\dim H^0(\mathcal{N}_{W/Z}^*\otimes\mathcal{L}_{|W})+\dim H^1(Z,\mathcal{L}^*).
\label{what we are going to prove}
\end{equation}
 
\begin{lemme}\label{slemma: horrible}
In the hypotheses of Proposition \ref{prop:dimext},
$$
H^{1}(E\otimes\mathcal{L}^*)=H^{1}(\mathcal{L}^*)\ \text{and}\ H^{2}(E\otimes \mathcal{L}^*)=0.
$$
\end{lemme}
\begin{proof}
From the exact sequence 
$$\xymatrix{ 0\ar[r]&\mathcal{I}_{W}\ar[r] & \ar[r]\mathcal{O}_Z &\mathcal{O}_W  \ar[r]& 0},$$
we see that $H^{0}(\mathcal{I}_{W})=H^{1}(\mathcal{I}_{W})=H^{2}(\mathcal{I}_{W})=0.$
Hence, from (\ref{eq:PrincipalGene x L*}) we read
$$H^{1}(E\otimes\mathcal{L}^*)=H^{1}(\mathcal{L}^*)
\qand 
H^{2}(E\otimes\mathcal{L}^*)=H^{2}(\mathcal{L}^*).$$
Moreover, the latter is trivial by the hypotheses of the Hartshorne-Serre construction (Theorem \ref{thm: Hartshorne-Serre}). 
\end{proof}
Now, from the spectral sequence
$$E_{2}^{p,q}:=H^{p}(\Extr^{q}(\mathcal{I}_{W}\otimes \mathcal{L},E))\Rightarrow E_{n}:= \Ext^{n}(\mathcal{I}_{\ell}\otimes\mathcal{L},E),$$
we obtain  the following exact sequence:
\begin{equation}
\xymatrix{ 0\ar[r]& E_{2}^{1,0}\ar[r] & E^{1}\ar[r]&E_{2}^{0,1}\ar[r] & E_{2}^{2,0}.}
\label{spectralx}
\end{equation}
Moreover, we have 
$$
\Extr^{q}(\mathcal{I}_{W}\otimes \mathcal{L},E)=\Extr^{q}(\mathcal{I}_{W},\mathcal{O}_{Z})\otimes E\otimes\mathcal{L}^*,
$$
with 
$\Extr^{1}(\mathcal{I}_{W},\mathcal{O}_{Z})=\wedge^{2}\mathcal{N}_{W/Z}
$ (see \eg \cite[Section 1]{Arrondo}).
Hence, (\ref{spectralx}) provides
$$
\xymatrix@R0pt{ 
        0\ar[r]&H^{1}(E\otimes \mathcal{L}^*) \ar[r] &\Ext^{1}(\mathcal{I}_{W}\otimes \mathcal{L},E)\\
        \ar[r]&H^{0}(E_{|W}\otimes \wedge^{2}\mathcal{N}_{W/Z}\otimes\mathcal{L}_W^*)\ar[r] & H^{2}(E\otimes\mathcal{L}^*)
}
$$
and, by Lemma \ref{slemma: horrible}, we obtain:
$$\dim \Ext^{1}(\mathcal{I}_{W}\otimes \mathcal{L},E)=\dim H^{1}(\mathcal{L}^*)+\dim H^{0}(E_{|W}\otimes \wedge^{2}\mathcal{N}_{W/Z}\otimes\mathcal{L}_W^*).$$
Since $E$ is a Hartshorne--Serre bundle obtained from the line $W$, we have: $\wedge^{2}\mathcal{N}_{W/Z}\otimes\mathcal{L}_W^*=\mathcal{O}_W$, thus
$$\Ext^{1}(\mathcal{I}_{W}\otimes \mathcal{L},E)
 =H^{0}(E_{|W}).$$
 As explained in \cite[Section 1]{Arrondo}, restricting the exact sequence (\ref{PrincipalGene}) to $W$, we obtain  (\ref{what we are going to prove}) from 
$$
E_{|W}=\mathcal{N}_{W/Z}^*\otimes\mathcal{L}_{|W}
=\mathcal{N}_{W/Z}^*\otimes\wedge^{2}\mathcal{N}_{W/Z}
=\mathcal{N}_{W/Z}.
$$

\begin{rmk}\label{CalculN}
One way to determine $\mathcal{N}_{W/Z}$ is to find a surface $\mathcal{S}$ such that $W\subset \mathcal{S} \subset Z$, which fits in the  exact sequence:
\[
\xymatrix{ 
0\ar[r]&\ar[r]\mathcal{N}_{W/\mathcal{S}} & \mathcal{N}_{W/Z}\ar[r] & (\mathcal{N}_{\mathcal{S}/Z})_{|W} \ar[r]& 0
}.
\]
\end{rmk}

\begin{cor}\label{CInelastic2}
Let $(Z,\kd)$ be a building block, and let $E \to Z$ be an asymptotically stable Hartshorne--Serre bundle obtained from a genus $0$ curve $W\subset Z$ and a line bundle $\mathcal{L}\to Z$ as in Theorem \ref{thm: Hartshorne-Serre}. 
Let $\mcal^{s}_{\kd,\mathcal{A}}(v)$ be the moduli space of $\mathcal{A}$-$\mu$-stable bundles on $\kd$ with Mukai vector $v=v(E_{|\kd})$.

Suppose $H^1(E)=0$. Then $E$ is inelastic if and only if 
$$
\frac{1}{2}\dim \mcal^{s}_{\kd,\mathcal{A}}(v)
=
\dim H^0(W,\mathcal{N}_{W/Z})+\dim H^1(Z,\mathcal{L}^*)-\dim H^0(Z,E)+1.
$$
\end{cor}

\subsection{Inelasticity of \texorpdfstring{$E_+$}{E+} and \texorpdfstring{$E_-$}{E-}}
\label{sec: inelast E+}

We will now prove the inelasticity of the bundle $E_+$ constructed in Proposition \ref{prop:exa}, over the building block $Z_+$  obtained by blowing up $Y_+=\CP^1\times\CP^2$ from  Example \ref{ex:2conics}. 
For tidiness, we omit the $+$ subscript.

\begin{prop}\label{inelasticity}
Let $E\to Z=\Bl_\ccal Y$ be the bundle constructed in
Proposition \ref{prop:exa}, over the building block from Example \ref{ex:2conics}, satisfying:
\begin{enumerate}
\item
$c_{1}(E)=-\kd-G+H$,
\item
$E$ has a global section with vanishing locus given by an exceptional fibre   $\ell$ of $p_{1}:\wt{\ccal}\rightarrow \mathscr{C}$over the base locus of the anti-canonical pencil.
\end{enumerate}
Then the bundle $E$ is inelastic. 
\begin{proof}
By Corollary \ref{CInelastic2}, using Lemmata  \ref{lemma:Hart-Serre+} \ref{H1lem} and \ref{HE} with Proposition \ref{qui}, we only have to check that $\dim H^0(\mathcal{N}_{\ell/Z})=1$, but this follows directly from  the fact that the lines of class $\ell$ in $Z$ are parametrised by the curve $\ccal$.
\end{proof}
\end{prop}

\noindent NB.: In the light of Remark \ref{CalculN}, we could also verify that $\mathcal{N}_{\ell/Z}=\mathcal{O}_\ell\oplus\mathcal{O}_\ell(-1)$.

\label{sec: inelast E-}

\bigskip
Similarly, we prove the inelasticity of the bundle $E_-$ constructed in Proposition \ref{prop:exa2}, over the building block $Z_-$ obtained by blowing up  $Y_-\overset{2:1}{\longrightarrow}\mathbb{P}^1\times\mathbb{P}^2$ from Example  \ref{ex:2conics2}. We also omit the $-$ subscript.

\begin{prop}\label{inelasticity2}
Let $E\to Z$ be the bundle constructed in
Proposition \ref{prop:exa2}, over the building block from  Example \ref{ex:2conics2}, satisfying:
\begin{enumerate}
\item
$c_{1}(E)=G$,
\item
$E$ has a global section with vanishing locus $W$ such that $[W]=h$ (\cf Example \ref{ex:2conics2}).
\end{enumerate}
Then the bundle $E$ is inelastic. 
\begin{proof}
As before, with Corollary \ref{CInelastic2}, using Lemmata \ref{lemma:Hart-Serre-} \ref{H1lem2} and \ref{HE2} with Proposition \ref{qui} we have
to check that $\dim H^0(\mathcal{N}_{W/Z})= 2$,
which is true since the family of curves of class $h=\left[W\right]$ in $Z$ is parametrised by a double cover of $\CP^1\times \CP^1$ (see Section \ref{sec:semi-Fano Blow-up of P3}).
\end{proof}
\end{prop}

\noindent NB.: Using  Remark \ref{CalculN}, we could also see that $\mathcal{N}_{W/Z}=\mathcal{O}_W\oplus\mathcal{O}_W$.

\section{Proof of Theorem \ref{thm:main}}
\label{sec:wrap}

Twisting the Mukai vector 
$$
v'_{\kd_+}:=(2,\, 5A_+-3B_+,\, -18)
$$ 
by $\mathcal{O}_{\kd_+}(-2B_++3A_+)$ gives a natural isomorphism
$\mcal^{s}_{\kd_+,\mathcal{A}_+}(v_{\kd_+})\simeq\mcal^{s}_{\kd_+,\mathcal{A}_+}(v'_{\kd_+})$.
Moreover, since
$$
\mathcal{O}_{\kd_+}(-2B_++3A_+)=\mathcal{O}_{Z_+}(-2H_++3G_+)_{|\kd_+},
$$ 
we can rewrite Theorem \ref{thm:summaryintro} with
$\mcal^{s}_{\kd_+,\mathcal{A}_+}(v_{\kd_+'})$ instead of
$\mcal^{s}_{\kd_+,\mathcal{A}_+}(v_{\kd_+})$.

\begin{cor}
\label{cor:summarybis}
In the context of Example \ref{ex:2conics}, for every bundle
$\mathcal{G} \in  \mcal_{\kd_+,\mathcal{A}_+}^{s}(v_{\kd_+}')$ and
every complex line $V \subset H^1(\kd_+, \Endr_0(\mathcal{G}))$, there are
a smooth curve $\ccal_+ \in |{-}K_{Y_+|\kd_+}|$ and an asymptotically stable and inelastic vector bundle $F_+ \to Z_+$ with  
$$
(\rk , c_{1},c_{2})(F_+)=(2,5G_+-3H_+,\ell_++(H_+-G_+)\cdot(-2H_++3G_+)+(-2H_++3G_+)^2),
$$ 
such that $F_{+|\kd_+} = \mathcal{G}$ and
$\res : H^1(Z_+, \Endr_0(F_+)) \to H^1(\kd_+, \Endr_0(\mathcal{G}))$ has image $V$.
\end{cor}
We have a similar result on the block $(Z_-,\kd_-)$. Twisting the vector  $$
v'_{\kd_-}:=(2,\, 5A_--2B_-,\, -18)
$$ 
by $\mathcal{O}_{\kd_-}(-B_-+2A_-)$ identifies
$\mcal^{s}_{\kd_-,\mathcal{A}_-}(v_{\kd_-})\simeq\mcal^{s}_{\kd_-,\mathcal{A}_-}(v'_{\kd_-})$ and, since
$$
\mathcal{O}_{\kd_+}(-B_-+2A_-)=\mathcal{O}_{Z_-}(-H_-+2G_-)_{|\kd_-},
$$ 
we can reformulate Propositions \ref{prop:exa2}, \ref{stabi2} and \ref{inelasticity2} for 
$$
F_-:=E_-\otimes\mathcal{O}_{Z_-}(-H_-+2G_-).
$$
\begin{cor}
\label{cor:summarybis2}
In the context of Example \ref{ex:2conics2}, there exists a family of asymptotically stable and inelastic vector bundles $\{F_- \to Z_-\}$, parametrised by the set of the lines in $Y_-$ of class $h_-$, such that 
$F_{-|\kd_-}\in\mcal_{\kd_-,\mathcal{A}_-}^{s}(v_{\kd_-}')$ and  
$$(\rk , c_{1},c_{2})(F_-)=(2,5G_--2H_-,h_-+G_-\cdot(-H_-+2G_-)+(-H_-+2G_-)^2).$$
\end{cor}

Theorem \ref{thm:main} is immediate from the following result, which we deduce
from Corollaries \ref{cor:summarybis} and \ref{cor:summarybis2}.

\begin{thm}
\label{thm:transv}
Let $\hkr : \kd_+ \to \kd_-$ be a matching between $Y_+=\CP^1\times\CP^2$ and $Y_-\overset{2:1}{\longrightarrow}\mathbb{P}^1\times\mathbb{P}^2$ .
Then there exist smooth curves $\ccal_\pm \in |{-}K_{Y_\pm | \kd_\pm}|$ and holomorphic bundles $F_\pm \to Z_\pm$
over the resulting blocks $Z_\pm := \Bl_{\ccal_\pm} Y_\pm$, with 
\begin{align*}
(\rk,c_{1},c_{2})(F_+)&=(2,5G_+-3H_+,\ell+(H_+-G_+)\cdot(-2H_++3G_+)+(-2H_++3G_+)^2)\\
(\rk,c_{1},c_{2})(F_-)&=(2,5G_--2H_-,h_-+G_-\cdot(-H_-+2G_-)+(-H_-+2G_-)^2),
\end{align*} 
satisfying all the hypotheses of Theorem \ref{thm:HenriqueThomas}.
\end{thm}

\begin{proof}
We fix a representative $F_- \to Z_-$ in the family of holomorphic bundles  from Corollary  \ref{cor:summarybis2},
to be matched by a bundle $F_+ \to Z_+$ given by Corollary \ref{cor:summarybis}, so that asymptotic stability and inelasticity hold from the outset.

It remains to address compatibility and transversality.
Since the chosen configuration for $\hkr$ ensures that
$\hkr^*$ identifies the Mukai vectors of $F_{\pm|\kd_\pm}$, 
it induces a map $\bar \hkr^* :\mcal_{\kd_-,\mathcal{A}_-}^{s}(v'_{\kd_-})\to \mcal_{\kd_+,\mathcal{A}_+}^{s}(v'_{\kd_+})$.
In particular, the target moduli space is $2$-dimensional, by Proposition \ref{qui}, and $\hkr^* (\image \res_-)$ is  $1$-dimensional, since the bundles $\{F_-\}$ are parametrised by lines of fixed class $h_-$. So indeed we apply Corollary \ref{cor:summarybis} with $\mathcal{G}= \bar \hkr^* (F_{-|\kd_-})$ and any choice of a direct complement subspace $V$ such that 
$$
V\oplus \bar
\hkr^* (\image \res_-)=H^1(\kd_+, \Endr_0(\bar \hkr^* (F_{-|\kd_-}))).
$$

Denoting by $\mscr_{\kd_\pm}(v)$ the moduli space of ASD instantons over
$\kd_\pm$ with Mukai vector $v$,
the maps $f_\pm$ in Theorem \ref{thm:HenriqueThomas} (\cf (\ref{eq: isomorphism f})) are the linearisations of the Hitchin-Kobayashi isomorphisms 
$$
\mcal^{s}_{\kd_\pm,\mathcal{A}_\pm}(v'_{\kd\pm})
\simeq
\mscr_{\kd_\pm}(v'_{\kd\pm}).
$$
Therefore, our bundles  $F_\pm$ indeed satisfy $A_{\infty,+}=\bar \hkr^* A_{\infty,-}$ for the corresponding
instanton connections. Moreover, by linearity,   $\lambda_+(H^1(Z_+, \Endr_0(F_+)))$ is transverse in
$T_{A_{\infty,+}} \mscr_{\kd_+}(v_{\kd_+}')$
 to the image of the real $2$-dimensional subspace  $\lambda_-(H^1(Z_-,\Endr_0(F_-))) \subset
T_{A_{\infty,-}} \mscr_{\kd_-}(v_{\kd_-}')$ under the linearisation of $\bar
\hkr^*$.
\end{proof}

\begin{rmk}
Similar techniques could still be used on blocks with a perpendicular lattice $N_0^\bot$ of rank higher than $2$. Indeed,
according to Propositions \ref{FanoBlock} and \ref{6.18},
we can choose K\"ahler classes $\kclass_\pm$ on $Z_\pm$ such that
the restrictions $\kclass_{\pm|\kd_\pm}$ are arbitrarily close to
$\mathcal{A}_{\pm}$. Hence it is not a problem to
consider asymptotic stability with respect to $\mathcal{A}_\pm$ instead
of $\kclass_\pm$.
\end{rmk}

\pagebreak[2]

\end{document}